 \documentclass[11pt]{amsart}
 \usepackage{amsmath,amsfonts,mathrsfs,amsthm,amssymb}
 \usepackage{amssymb}
 \usepackage{graphicx}
 \usepackage{framed}
 \usepackage{color}
 \usepackage{hyperref}
 \usepackage{mathabx}
 \usepackage{mathtools}
 \usepackage{tikz}
 \usepackage{pdflscape}
 
\newcommand\diag{\operatorname{diag}}

 \newcommand\id{\mathrm{id}}
 \newcommand\tr{\mathrm{tr}}
 
 \newcommand\tGr{\mathrm{Gr}}
 \newcommand\tLG{\mathrm{LG}}

 \newcommand\tMat{\mathrm{Mat}}

 \newcommand\qRa{\quad\Rightarrow\quad}

 \newcommand\bop{\bigoplus}
 \newcommand\op{\oplus}

 \newcommand\fann{\mathfrak{ann}}

 \newcommand\bS{{\mathbf S}} 
 \newcommand\bT{{\mathbf T}}
 \newcommand\bU{{\mathbf U}}
 \newcommand\bV{{\mathbf V}}
   
 \newcommand\bX{{\mathbf X}} 
 \newcommand\bY{{\mathbf Y}}
 \newcommand\bZ{{\mathbf Z}}
 
 \newcommand\rnk{\mathrm{rank}}

 \newcommand\fcsp{\mathfrak{csp}}

 \newcommand\ff{{\mathfrak f}}
 \newcommand\fg{{\mathfrak g}}
 \newcommand\fgl{\mathfrak{gl}}
 \newcommand\fh{{\mathfrak h}}

 \newcommand\fl{{\mathfrak l}}
 \newcommand\fm{{\mathfrak m}}

 \newcommand\fp{{\mathfrak p}}
 \newcommand\fq{{\mathfrak q}}
 \newcommand\fr{{\mathfrak r}}
 \newcommand\fs{{\mathfrak s}}
 \newcommand\fsl{\mathfrak{sl}}
 \newcommand\fso{\mathfrak{so}}
 \newcommand\fsp{\mathfrak{sp}}

 \newcommand\fz{{\mathfrak z}}

 \newcommand\fC{{\mathfrak C}}
 \newcommand\fD{{\mathfrak D}}

 \newcommand\fG{{\mathfrak G}}

 \newcommand\fP{{\mathfrak P}}

 \newcommand\fS{{\mathfrak S}}

 \newcommand\cC{{\mathcal C}}
 \newcommand\cD{{\mathcal D}}
 \newcommand\cE{{\mathcal E}}
 \newcommand\cF{{\mathcal F}}
 \newcommand\cG{{\mathcal G}}
 
 \newcommand\cI{{\mathcal I}}
 \newcommand\cJ{{\mathcal J}}
 
 \newcommand\cL{{\mathcal L}}
 \newcommand\cM{{\mathcal M}}
 \newcommand\cN{{\mathcal N}}

 \newcommand\cQ{{\mathcal Q}}
 \newcommand\cR{{\mathcal R}}
 \newcommand\cS{{\mathcal S}}
 \newcommand\cT{{\mathcal T}}
 
 \newcommand\cV{{\mathcal V}}
 \newcommand\cW{{\mathcal W}}

 \newcommand\sfa{\mathsf{a}}
 \newcommand\sfb{\mathsf{b}}
 \newcommand\sfc{\mathsf{c}}
 \newcommand\sfd{\mathsf{d}}
 \newcommand\sfe{\mathsf{e}}

 \newcommand\sfr{\mathsf{r}}
 \newcommand\sfs{\mathsf{s}}

 \newcommand\sfw{\mathsf{w}}

 \newcommand\sfC{\mathsf{C}}
 
 \newcommand\sfE{\mathsf{E}}
 \newcommand\sfF{\mathsf{F}}
 
 \newcommand\sfH{\mathsf{H}}
 \newcommand\sfI{\mathsf{I}}

 \newcommand\sfQ{\mathsf{Q}}

 \newcommand\sfY{\mathsf{Y}}
 \newcommand\sfZ{\mathsf{Z}}

 \newcommand\bbA{{\mathbb A}}
 
 \newcommand\bbC{{\mathbb C}}

 \newcommand\bbH{{\mathbb H}}

 \newcommand\bbO{{\mathbb O}}
 \newcommand\bbP{{\mathbb P}}
 \newcommand\bbQ{{\mathbb Q}}
 \newcommand\bbR{{\mathbb R}}

 \newcommand\bbV{{\mathbb V}}

 \newcommand\bbZ{{\mathbb Z}}

 \newcommand\tEnd{\mathrm{End}}

 \newcommand\vol{\mathrm{vol}}

 \newcommand\tspan{\mathrm{span}}

 \newcommand\Ben{\begin{enumerate}}
 \newcommand\Een{\end{enumerate}}
 \newcommand\Bex{\begin{example}}
 \newcommand\Eex{\end{example}}
 \def\inj{\hookrightarrow}

\def\tGL{\text{GL}}

 \newcommand\SL{\mathrm{SL}}

 \newcommand\SO{\mathrm{SO}}
 
 \newcommand\Sp{\mathrm{Sp}}
 \newcommand\ad{{\rm ad}}

 \newcommand\tgr{\mathrm{gr}}

 \def\assoc/{associative}
 \def\arb/{arbitrary}
 \def\btw/{between}
 \def\coeff/{coefficient}
 \def\cohom/{cohomology}
 \def\coord/{coordinate}
 \def\coordsys/{coordinate system}
 \def\cpt/{compact}
 \def\cred/{completely reducible}
 \def\cts/{continuous}
 \def\dga/{differential-graded algebra}
 \def\dR/{de Rham}
 \def\Euc/{Euclidean} 
 \def\grp/{group}
 \def\hom/{homomorphism}
 \def\inv/{invariant}
 \def\iso/{isomorphism}
 \def\La/{Lie algebra}
 \def\Lag/{Lagrangian Grassmannian}
 \def\LG/{Lie group}
 \def\MA/{Monge--Amp\`ere}
 \def\MC/{Maurer--Cartan}
 \def\lintr/{linear transformation} 
 \def\mfld/{manifold}
 \def\nb/{normal bundle}
 \def\nbd/{neighbourhood}
 \def\nondeg/{non-degenerate}
 \def\posdef/{positive definite}
 \def\pu/{partition of unity}
 \def\rep/{representation}
 \def\Riem/{Riemannian}
 \def\sg/{subgroup}
 \def\ss/{semi-simple}
 \def\inv/{invariant}
 \def\irr/{irreducible}
 \def\Jacid/{Jacobi identity}
 \def\li/{linearly independent}
 \def\nd/{nowhere dependent}
 \def\nz/{nowhere zero}
 \def\on/{orthonormal}
 \def\onb/{\on/ basis}
 \def\orc/{\orth/ complement}
 \def\orth/{orthogonal}
 \def\orp/{\orth/ projection}
 \def\pde/{partial differential equation}
 \def\resp/{respectively}
 \def\seq/{sequence}
 \def\std/{standard}
 \def\SW/{Stiefel-Whitney}
 \def\uc/{universal cover}
 \def\vb/{vector bundle}
 \def\vf/{vector field}
 \def\vs/{vector space}
 \def\wrt/{with respect to}
 
 \renewcommand\mod{\,{\rm mod}\ }
 \newcommand\qbox[1]{\quad\mbox{#1}\quad}
 \renewcommand\dim{{\rm dim}}
 \newcommand\codim{{\rm codim}}
\newcommand\Seg{\operatorname{Seg}}
\newcommand\CSp{\operatorname{CSp}}
\newcommand\JS{\mathcal{J\!S}}
\newcommand\Ch{\operatorname{Ch}}
\newcommand\Ss{\operatorname{ss}}
 \newcommand\bE{{\mathbf E}}
 \newcommand\GJ{\cJ_3(\emptyset)}

\newtheorem{theorem}{Theorem}[section]
\newtheorem{lemma}[theorem]{Lemma}
\newtheorem{cor}[theorem]{Corollary}
\newtheorem{prop}[theorem]{Proposition}
\theoremstyle{defn}
\newtheorem{defn}[theorem]{Definition}
\newtheorem{example}[theorem]{Example}

\theoremstyle{remark}
\newtheorem{remark}[theorem]{Remark}
\numberwithin{equation}{section}

\title[Exceptionally simple PDE]
      {Exceptionally simple PDE}

\author[D.~The]{Dennis~The}

\date{\today}

\address{Fakult\"at f\"ur Mathematik, Universit\"at Wien, Oskar-Morgenstern-Platz 1, 1090 Wien, Austria}

\email{dennis.the@univie.ac.at}

\address{Department of Mathematics \& Statistics, University of Troms\o, N-9037, Troms\o, Norway}

\email{dennis.the@uit.no}



\begin{document}
\maketitle

 \begin{abstract} 
 We give local descriptions of parabolic contact structures and show how their flat models yield explicit PDE having symmetry algebras isomorphic to all complex simple Lie algebras except $\fsl_2$.  This yields a remarkably uniform generalization of the Cartan--Engel models from 1893 in the $G_2$ case.  We give a formula for the harmonic curvature of a $G_2$-contact structure and describe submaximally symmetric models for general $G$-contact structures.
 \end{abstract}

 \section{Introduction}
 \label{S:intro}

 The Cartan--Killing classification of all complex simple Lie algebras was one of the great milestones of 19th century mathematics.  In addition to the classical series of type $A_\ell,B_\ell,C_\ell,D_\ell$ (corresponding to the complex matrix Lie algebras $\fsl_{\ell+1}, \fso_{2\ell+1}, \fsp_{2\ell}, \fso_{2\ell}$), five surprising ``exceptional'' Lie algebras of type $G_2, F_4, E_6, E_7, E_8$ of dimensions 14, 52, 78, 133, 248 were discovered.  Since Lie algebras arose from the study of transformation groups, one can naturally ask for geometric structures whose symmetry algebra is a given simple Lie algebra.  In 1893, Cartan \cite{Car1893} and Engel \cite{Eng1893} announced the first explicit (local) geometric realizations for $G_2$ (see Table \ref{F:G2-models}), most of which can be formulated as differential equations.
 
 Later, in his 5-variables paper \cite{Car1910}, Cartan established remarkable correspondences between:
  \begin{itemize}
  \item contact (external) symmetries of (non-Monge-Amp\`ere) parabolic Goursat PDE in the plane;
  \item contact (external) symmetries of nonlinear involutive pairs of PDE in the plane;
  \item symmetries of $(2,3,5)$-distributions.
  \end{itemize}
 In a tour-de-force application of his method of equivalence, Cartan then solved the equivalence problem for $(2,3,5)$-distributions.  Nowadays, we formalize this as a (regular, normal) parabolic geometry of type $(G_2, P_1)$.  (For the parabolic subgroup $P_1 \subset G_2$, see ``Conventions'' below.)  This yields a notion of curvature for such geometries and there is a (locally) unique ``flat'' model with maximal symmetry dimension $\dim(G_2) = 14$.  The 1893 $G_2$-models $\overline\cE,\cE,\cF$ are associated to the flat case of this general curved story.
 
  \begin{table}[h]
 \[
 \begin{array}{|c|c|c|} \hline
 \mbox{Dim} & \mbox{Geometric structure} & \mbox{Model} \\ \hline\hline
 7 & \mbox{Parabolic Goursat PDE $\cF$} & \begin{array}{l} 9(u_{xx})^2 + 12 (u_{yy})^2 (u_{xx} u_{yy} - (u_{xy})^2) \\
 \qquad + 32 (u_{xy})^3 - 36u_{xx} u_{xy} u_{yy} = 0\end{array}\\ \hline
 6 & \mbox{Involutive pair of PDE $\cE$} & u_{xx} = \frac{1}{3} (u_{yy})^3, \quad u_{xy} = \frac{1}{2} (u_{yy})^2 \\ \hline
 5 & \mbox{$(2,3,5)$-distribution $\overline\cE$} & 
 \begin{array}{c} 
 dx_2 - x_4 dx_1, \,\,
 dx_3 - x_2 dx_1, \,\,
 dx_5 - x_4 dx_2
\\
  \mbox{(equivalently, Hilbert--Cartan: $Z' = (U'')^2$)}
  \end{array}\\ \hline
 5 & \begin{tabular}{c} $G_2$-contact structure \\ (contact twisted cubic field) \end{tabular} & \left\{ \begin{array}{c} 
 dz + x_1 dy_1 - y_1 dx_1 + x_2 dy_2 - y_2 dx_2 = 0,\\ 
 dx_2^2 + \sqrt{3} dy_1 dy_2 = 0,\\
 dx_2 dy_2 - 3 dx_1 dy_1= 0,\\
 dy_2^2 + \sqrt{3} dx_1 dx_2 = 0
 \end{array} \right.\\ \hline
 \end{array}
 \]
 \caption{The Cartan--Engel $G_2$ models}
 \label{F:G2-models}
 \end{table}

Yamaguchi \cite{Yam1999} generalized the reduction theorems underlying Cartan's correspondences in \cite{Car1910, Car1911}.  For all $G \neq A_\ell, C_\ell$, he identified the reduced geometries analogous to $G_2 / P_1$ (see \cite[pg.310]{Yam1999}) and proved the existence of corresponding (nonlinear) PDE admitting external symmetry $\fg$.  However, these PDE were not {\em explicitly} described.\footnote{In \cite[Sec.6.3]{Yam2009}, Yamaguchi gave explicit {\em linear} PDE with $E_6$ and $E_7$ symmetry, but these are not the PDE from \cite{Yam1999}.}  Exhibiting these models is one of the results of our article.

 Notably absent in the Cartan--Yamaguchi story is Engel's 1893 model, namely a contact 5-manifold whose contact distribution is endowed with a twisted cubic field, which is the flat model for {\em $G_2$-contact structures}, i.e.\ $G_2 / P_2$ geometries.  Our article will focus on its generalization to structures called {\em $G$-contact structures} (or {\em parabolic contact structures}), modelled on the adjoint variety $G/P \cong G^{ad} \inj \bbP(\fg)$ of a (connected) complex simple Lie group $G$.  This adjoint variety is always a complex contact manifold except for $A_1 / P_1 \cong \bbP^1$, so $G = A_1 \cong \SL_2$ will be henceforth excluded.  Letting $\dim(G/P) = 2n+1$, a $G$-contact structure consists of a contact manifold $(M^{2n+1},\cC)$ (locally, the first jet-space $J^1(\bbC^n,\bbC)$) with $\cC$ (a field of conformal symplectic spaces) equipped with additional geometric data.

 Restrict now to $G \neq A_\ell, C_\ell$.  Earlier formulations of $G$-contact structures identified $\cC$ as a tensor product of one or more auxilliary vector bundles: in the $G_2$ case, $\cC \cong S^3 E$ where $E \to M$ has rank two,  and similarly for the exceptional cases \cite[\S 4.2.8]{CS2009}; for the $B_\ell,D_\ell$ cases (Lie contact structures), see \cite{SY1989}.  While these abstract descriptions were sufficient for solving the equivalence problem, no concrete local descriptions were given in these works.  Recently, a local description in terms of a conformal quartic tensor $[\sfQ]$ on $\cC$ was used by Nurowski \cite{Nur2013} and Leistner et al. \cite{LNS2016}.  But this viewpoint does not naturally lead to PDE.

 We start from Engel's algebro-geometric perspective:  $G$-contact structures can be described in terms of a {\em sub-adjoint variety field} $\cV \subset \bbP(\cC)$.  But $\cV$ naturally induces other fields $\widehat\cV \subset \widetilde\cV \subset M^{(1)}$ and $\tau(\cV) = \{ \sfQ = 0 \} \subset \bbP(\cC)$, and it turns out that these essentially give equivalent descriptions of the same $G$-contact structure.  In particular, their symmetry algebras are the {\em same}.  Here, $M^{(1)} \to M$ is the {\em Lagrange--Grassmann bundle}, whose fibre over $m \in M$ is the Lagrangian--Grassmannian $\tLG(\cC_m)$.  Locally, $M^{(1)}$ is isomorphic to the second jet-space $J^2(\bbC^n,\bbC)$, so $\widehat\cV$ and $\widetilde\cV$ yield second-order PDE $\cE$ and $\cF$.  (Note $\cE \subset \cF$.)  Since the equivalence problem for $G$-contact structures is solved (see \cite{CS2009} for details) via a (regular, normal) parabolic geometry of type $(G,P)$, the maximal symmetry dimension is $\dim(G)$, and the (locally unique) flat $G$-contact structure realizes it.  In this way, the flat structure yields $G$-invariant PDE $\cE$ and $\cF$ (fibred over $M = G/P$) with ({\em external / contact}) symmetry algebra precisely $\fg$.

 To make $\cE$ and $\cF$ explicit locally, we use (see \S \ref{S:LM}) the parametric description of a sub-adjoint variety due to Landsberg and Manivel \cite{LM2001} in terms of a (complex) Jordan algebra $W$ with cubic form $\fC \in S^3 W^*$.  Let $n = 1 + \dim(W)$. Pick any basis $\{ \sfw_a \}_{a=1}^{n-1}$ on $W$ (with dual basis $\{ \sfw^a \}$) and let $\{ x^i \}_{i=0}^{n-1}$ be corresponding linear coordinates adapted to $\bbC^n \cong \bbC \op W$.  Extend this to standard jet-space coordinates $(x^i,u,u_i,u_{ij})$ on $J^2(\bbC^n,\bbC)$.  Then Theorem \ref{T:flat} gives our generalization in a uniform manner (see Tables \ref{F:ExcSimp} and \ref{F:sym}).  
 
  \begin{table}[h]
 \[
 \begin{array}{|c|c|c|} \hline
 \cF \subset J^2(\bbC^n,\bbC) & \left\{ \begin{array}{l} u_{00} = t^a t^b u_{ab} - 2\fC(t^3)\\ u_{a0} = t^b u_{ab} - \frac{3}{2} \fC_a(t^2)\end{array} \right.\\ \hline
 \cE \subset J^2(\bbC^n,\bbC) & (u_{ij}) 
 = \displaystyle\left( \begin{array}{c|c} u_{00} & u_{0b} \\ \hline u_{a0} & u_{ab}\end{array} \right) 
 = \left( \begin{array}{c|c} \fC(t^3) & \frac{3}{2} \fC_b(t^2)\\[0.05in] \hline \frac{3}{2} \fC_a(t^2) & 3\fC_{ab}(t)  \end{array} \right) \\ \hline
 (J^1(\bbC^n,\bbC),\cC,\cV) 
 & \begin{array}{c} 
 \cV = \{ [\bV(\lambda,t)] : [\lambda,t] \in \bbP(\bbC \op W) \} \subset \bbP(\cC), \mbox{ where} \\
 \bV(\lambda,t) = \lambda^3\bX_0 - \lambda^2 t^a \bX_a - \frac{1}{2} \fC(t^3) \bU^0 - \frac{3}{2} \lambda \fC_a(t^2) \bU^a,\\
 \mbox{and}\quad \bX_i = \partial_{x^i} + u_i \partial_u, \quad \bU^i = \partial_{u_i} \end{array}\\ \hline
 (J^1(\bbC^n,\bbC),\cC,[\sfQ])
 & \begin{array}{c}
 \sfQ = (\omega^i \theta_i )^2 + 2 \theta_0 \fC(\Omega^3) - 2 \omega^0 \fC^*(\Theta^3) - 9 \fC_a(\Omega^2) (\fC^*)^a(\Theta^2),\\
 \mbox{where }\quad 
 \omega^i = dx^i, \,\, 
 \theta_i = du_i, \,\,
 \Omega = \omega^a \otimes \sfw_a, \,\,
 \Theta = \theta_a \otimes \sfw^a 
 \end{array}\\ \hline
 \end{array}
 \]
 \begin{align*}
  (t = t^a \sfw_a\in W;\,\, n = 1 + \dim(W);\,\, 0 \leq i,j \leq n-1; \,\,1 \leq a,b \leq n-1)
 \end{align*}
 \caption{Equivalent descriptions of the flat $G$-contact structure $(G \neq A_\ell, C_\ell)$}
 \label{F:ExcSimp}
 \end{table}

 \begin{small}
 \begin{table}[h]
 \[
 \begin{array}{|c||@{\,}c@{\,}|@{\,}c@{\,}|@{\,}c@{\,}|@{\,}c@{\,}|@{\,}c@{\,}|@{\,}c@{\,}|@{\,}c@{\,}|@{\,}c@{\,}|} \hline
 G & B_\ell \,\, (\ell \geq 3) & D_\ell \,\, (\ell \geq 5) & G_2 & D_4 & F_4 & E_6 & E_7 & E_8\\ \hline
 \mbox{Cubic Jordan algebra } W & \JS_{2\ell-5} & \JS_{2\ell-6} & \GJ & \cJ_3(\underline{0}) & \cJ_3(\bbR_\bbC) & \cJ_3(\bbC_\bbC)  & \cJ_3(\bbH_\bbC)  & \cJ_3(\bbO_\bbC) \\ \hline
 n = 1 + \dim(W) & 2\ell - 3 & 2\ell - 4 & 2 & 4 & 7 & 10 & 16 & 28\\ \hline
 \begin{tabular}{c} Model $G/P$ \& \\  $\dim(G/P) = 2n+1$ \end{tabular} 
 & \begin{array}{c} B_\ell / P_2 \\ 4\ell - 5 \end{array}
 & \begin{array}{c} D_\ell / P_2 \\ 4\ell - 7 \end{array}
 & \begin{array}{c} G_2 / P_2 \\ 5 \end{array} 
 & \begin{array}{c} D_4 / P_2 \\ 9 \end{array} 
 & \begin{array}{c} F_4 / P_1 \\ 15 \end{array} 
 & \begin{array}{c} E_6 / P_2 \\ 21 \end{array} 
 & \begin{array}{c} E_7 / P_1 \\ 33 \end{array} 
 & \begin{array}{c} E_8 / P_8 \\ 57 \end{array}
 \\ \hline
 \begin{tabular}{c} CC reduction $\overline\cE$ of $\cE$\\ \& $\dim(\overline\cE) = 3n-1$\end{tabular} 
 & \begin{array}{c} B_\ell / P_{1,3} \\ 6\ell - 10 \end{array}
 & \begin{array}{c} D_\ell / P_{1,3} \\ 6\ell - 13 \end{array}
 & \begin{array}{c} G_2 / P_1 \\ 5 \end{array}
 & \begin{array}{c} D_4 / P_{1,3,4} \\ 11 \end{array}
 & \begin{array}{c} F_4 / P_2 \\ 20 \end{array} 
 & \begin{array}{c} E_6 / P_4 \\ 29 \end{array} 
 & \begin{array}{c} E_7 / P_3 \\ 47 \end{array} 
 & \begin{array}{c} E_8 / P_7 \\ 83 \end{array} \\ \hline
 \end{array}
 \]
 \caption{Data associated with the flat $G$-contact structure $(G \neq A_\ell, C_\ell)$}
 \label{F:sym}
 \end{table}
 \end{small}
 
 Remarkably, the PDE $\cF$ and $\cE$ admit an even simpler description: they are respectively the first and second-order envelopes of the family of inhomogeneous {\em linear} PDE $u_{00} - 2t^a u_{a0} + t^a t^b u_{ab} = \fC(t^3)$ parametrized by $t = t^a \sfw_a \in W$, i.e.\ a (generalized) {\em Goursat parameterization}.

 Computing symmetries of PDE \cite{Olv1995,KLR2007} is algorithmic, but it is virtually impossible for most of our PDE $\cE$ and $\cF$ using standard techniques (even with the aid of computer algebra).  In stark contrast, symmetries of $\cV$ can be efficiently computed by-hand (Theorem \ref{T:sym-flat}) and uniform formulas for $\fg$ represented as contact vector fields are given in Table \ref{F:explicit-sym}.  These make explicit some statements made in \cite{Car1893-German}, e.g.\ Cartan briefly writes:
 {\em ``Endlich habe ich eine einfache 248-gliedrige Ber\"uhrungstransformationsgruppe $G_{248}$ in $R_{29}$ gefunden.''}  (Cartan is actually referring to a representation of $E_8$ on the 57-dimensional contact manifold $E_8 / P_8$; this has local coordinates $(x^i,u,u_i)$, and $R_{29}$ refers to the coordinates $(x^i,u)$, despite the fact that there is no natural fibration.)  Our  formulas generalize those of \cite[\S 4.4]{LNS2016} for $G_2$ and $B_3$ obtained via $[\sfQ]$.  Similar uniform descriptions appeared in work of G\"unaydin and Pavlyk \cite[\S 4.1]{GP2009}.  Our approach identifies a rich geometric / PDE perspective underlying these descriptions.

The canonical distribution $\cC^{(1)}$ on $M^{(1)}$ induces a distribution $\cD$ on $\cE$.  The tableau associated to $(\cE,\cD)$ is involutive (in the sense of Cartan--K\"ahler) only in the $G_2$ or $B_3$ cases (Theorem \ref{T:involutive}). 
 Also, $(\cE,\cD)$ has infinite-dimensional ({\em internal}) symmetry algebra because of a rank one distribution $\Ch(\cD)$ of {\em Cauchy characteristics}, i.e.\ symmetries of $(\cE,\cD)$ contained in $\cD$ itself.  The (local) leaf space $\overline\cE = \cE / \Ch(\cD)$ inherits a distribution $\overline\cD$ (see \eqref{E:D-red}), which can be expressed\footnote{The expression \eqref{E:Monge} is only a mnemonic device: ``symmetries'' refer to internal symmetries of $(\overline\cE,\overline\cD)$, independent of $J^{1,2}$.} as the mixed order, vector PDE $\overline\cE \subset J^{1,2}(\bbC^{n-1},\bbC^2)$:
 \begin{align} \label{E:Monge}
 Z_a = \frac{3}{2} \fC_a(T^2), \quad U_{ab} = 3\fC_{ab}(T), \qbox{where} T \in W.
 \end{align}
 Here, we regard $Z,U$ as functions of $X^a$, and $Z_a$, $U_{ab}$ refer to $\frac{\partial Z}{\partial X^a}$, $\frac{\partial^2 U}{\partial X^a \partial X^b}$.  The PDE \eqref{E:Monge} provides a fifth model with symmetry $\fg$, and generalizes the Hilbert--Cartan equation in the $G_2$ case, which is a {\em second-order Monge equation}.  (See \cite{ANN2015} for Monge geometries of {\em first-order}.)  All solutions to \eqref{E:Monge} are given in \S \ref{S:Monge}, and these lead to solutions of $\cE$.
Involutivity in the $G_2$ or $B_3$ cases leads to solutions depending on one or two {\em functions} of one variable respectively, but only on arbitrary {\em constants} in the general case.
  
 While the PDE $\cF$ and $\cE$ in the flat case are indeed those implicitly referred to by Yamaguchi, this is a priori not clear since we obtained these in a completely different manner via fibrewise constructions on $\cV$.  This is discussed in \S \ref{S:Cauchy} and \S \ref{S:Goursat} where the associated reduction theory is illustrated in detail.  In particular, $(\overline\cE,\overline\cD)$ is the flat model for the reduced geometries identified by Yamaguchi.  Most of these geometries are {\em rigid}: only the $G_2$ and $B_3$ cases admit curved deformations.

 Following our initial arXiv post of this article, other (hypersurface) PDE with symmetry $\fg$, alternative to our $\cF$, were found \cite{AGMM2016}.  While these are equivalent representations of the flat $G$-contact structure, their relationship to the sub-adjoint variety field $\cV$ is unclear.  Uncovering such natural geometric constructions would allow these new PDE to be written explicitly, analogous to what we have done here.  The reduction theory for these PDE would be an interesting topic for investigation.

 In \S \ref{S:degenerate}, we discuss the geometry associated with the exceptional type $A$ and $C$ cases.  We have:
 \begin{itemize}
 \item $u_{ij} = 0,\,\, 1 \leq i,j \leq n$ has point symmetry $A_{n+1}$, i.e.\ the flat $A_{n+1}$-contact structure.
 \item $u_{ijk} = 0,\,\, 1 \leq i,j,k \leq n$ has contact symmetry $C_{n+1}$, i.e.\ the flat $C_{n+1} / P_{1,n+1}$ structure.
 \end{itemize}
 Via a twistor correspondence \cite{Cap2005}, the latter can be viewed as the flat $C_{n+1}$-contact structure.  Indeed, all (complex) parabolic contact structures admit a description in terms of PDE.
  
 All $G$-contact structures are non-rigid geometries and we briefly discuss the non-flat case in \S \ref{S:non-flat}.  For $G_2$-contact structures (\S \ref{S:G2P2}), we give a formula for the harmonic curvature and give some symmetry classification results. We then conclude with some submaximally symmetric models in the general case (\S \ref{S:submax-sym} and \S \ref{S:AC-submax}).  In general, the PDE $\cE$ and $\cF$ for non-flat $G$-contact structures do not satisfy the Cartan--Yamaguchi reduction criteria, which explains the absence of $\cV$ in their story.
 
 \section*{Acknowledgements}
 
 We thank I.\ Anderson, R.\ Bryant, L.\ Manivel, and K.\ Sagerschnig for helpful discussions.  An extremely useful tool throughout this project was the {\tt DifferentialGeometry} package in {\tt Maple}.  D.T. was supported by a Lise Meitner Fellowship (project M1884-N35) of the Austrian Science Fund (FWF).
 
 \section*{Conventions}  
 
 We will work exclusively with complex Lie groups and Lie algebras, complex manifolds and jet-spaces, etc.  (However, all our results are analogously true for split-real forms.)
 
 Given a rank $\ell$ complex simple Lie algebra $\fg$, a Borel subalgebra is assumed fixed.  Let $\fh$ be the Cartan subalgebra, root system $\Delta \subset \fh^*$, simple roots $\Delta^0 = \{ \alpha_i \}_{i=1}^\ell$ (use the Bourbaki / {\tt LiE} ordering), and dual basis $\{ \sfZ_i \}_{i=1}^\ell \subset \fh$.  Let $\fg_\alpha$ be the root space for $\alpha \in \Delta$.  Let $\{ \lambda_i \}_{i=1}^\ell$ be the fundamental weights.
 
 A parabolic subalgebra $\fp \subset \fg$ is marked by crosses on the nodes $I_\fp = \{ i : \fg_{-\alpha_i} \not\subset \fp \} \subset \{ 1, ..., \ell \}$ of the Dynkin diagram of $\fg$.  A parabolic subgroup $P \subset G$ with Lie algebra $\fp$ is denoted by $P_{I_\fp}$.  (For the closed $G$-orbit $G/P  \inj \bbP(\bbV)$, where the $G$-irrep $\bbV$ has highest weight $\lambda = \sum_{i=1}^\ell r_i(\lambda) \lambda_i$, we have $I_\fp = \{ i : r_i(\lambda) \neq 0 \}$.)  The grading element $\sfZ = \sum_{i \in I_\fp} \sfZ_i$ gives a grading $\fg = \bop_{k=-\nu}^\nu \fg_k$, where $\fg_k = \bigoplus_{\sfZ(\alpha) = k} \fg_\alpha$, with $\fp = \fg_{\geq 0}$.
 
 \section{Sub-adjoint varieties and natural constructions}
 \label{S:SA}
 
 \subsection{The Lagrangian-Grassmannian}
 \label{S:LG}
 
 Let $n \geq 2$ and let $(V,\eta)$ be a $2n$-dimensional symplectic vector space.  A subspace $L \subset V$ is {\em Lagrangian} if $\dim(L) = n$ and $\eta|_L \equiv 0$.  The Lagrangian-Grassmannian $\tLG(V)$ consists of all such subspaces and depends only on the conformal class $[\eta]$.  The Lie groups $\Sp(V)$ and $\CSp(V)$ consist of linear transformations of $V$ that preserve $\eta$ and $[\eta]$ respectively.  These act transitively on the manifold $\tLG(V)$.  Since $T_L( \tLG(V) ) \cong S^2 L^*$, then $\dim(\tLG(V)) = \binom{n+1}{2}$.
 
 A basis $\{ \sfe_1,..., \sfe_{2n} \}$ of $(V,\eta)$ is {\em conformal symplectic} (a ``{\em CS-basis}'') if $\eta$ is represented in this basis by a multiple of $\begin{psmallmatrix} 0 & \id_n\\ -\id_n & 0 \end{psmallmatrix}$.  Then 
 $\fsp_{2n} = \left\{ \begin{psmallmatrix} a & b\\ c & -a^\top\end{psmallmatrix} : a,b,c \in \tMat_{n\times n};\, b,c \mbox{ symmetric} \right\}$.
 Now $o = \tspan\{ \sfe_1, ..., \sfe_n \}$ has stabilizer $P_n =  \left\{ \begin{psmallmatrix} A & B \\ 0 & (A^\top)^{-1} \end{psmallmatrix} : A^{-1} B \mbox{ symmetric} \right\} \subset \Sp(V)$ with Lie algebra $\fp_n \subset \fsp_{2n}$.  We obtain {\em standard coordinates} about $o$ by mapping the symmetric matrix $X$ to $\tspan\{ \sfe_i + X_{ij} \sfe_{n+j} \}_{i=1}^n$.
 
 For $g = \begin{psmallmatrix} A & B \\ C & D \end{psmallmatrix} \in \CSp(V)$ near the identity,  $ \begin{psmallmatrix} A & B \\ C & D \end{psmallmatrix} \cdot \begin{psmallmatrix} I & 0 \\ X & I \end{psmallmatrix} / P_n = \begin{psmallmatrix} I & 0 \\ \widetilde{X} & I \end{psmallmatrix} / P_n$, where
 \begin{align} \label{E:Mobius}
 \widetilde{X} = (C+DX)(A+BX)^{-1}.
 \end{align}

 \subsection{Adjoint and sub-adjoint varieties}
 \label{S:SA-var}
 
 Let $G$ be a (connected) complex simple Lie group with Lie algebra $\fg$.  The unique closed $G$-orbit $G/P \cong G^{\ad} \inj \bbP(\fg)$ is the adjoint variety of $G$.  This is a complex contact manifold except when $G = A_1$ (henceforth excluded).  Otherwise, the reductive part $G_0 \subset P$ induces a $G_0$-invariant {\em contact grading} on $\fg$, induced by a grading element $\sfZ \in \fz(\fg_0)$ (see ``Conventions''): 
 \begin{itemize}
 \item $\fg = \fg_{-2} \op \fg_{-1} \op \fg_0 \op \fg_1 \op \fg_2$, where $\fp = \fg_{\geq 0}$ and $(\fg_{-k})^* \cong \fg_k$ for $k \neq 0$ (via the Killing form);
 \item $[\fg_i,\fg_j] \subset \fg_{i+j}$ for $i,j \in \bbZ$ (take $\fg_i = 0$ for $|i| > 2$);
 \item $\fg_-$ is a Heisenberg algebra, i.e.\ $\dim(\fg_{-2}) = 1$ and the bracket $\eta : \bigwedge^2 \fg_{-1} \to \fg_{-2}$ is non-degenerate.
 \end{itemize}
 In particular, $V = \fg_{-1}$ is a CS-vector space and $G_0 \subset \CSp(V)$.  We have that $V$ is $G_0$-irreducible iff $G \neq A_\ell$; also, $\fg_0 \neq \fcsp(V)$ iff $G \neq C_\ell$.
 
 For $G \neq A_\ell,C_\ell$, we have $\lambda = \lambda_j$ (i.e.\ $j$ is the ``contact node''), $P = P_j$ is maximal parabolic, and $\fg_0 = \fz(\fg_0) \op \fg_0^{\Ss}$ with $\fz(\fg_0)$ spanned by $\sfZ = \sfZ_j$.  The {\em sub-adjoint variety} $\cV$ for $G$ is the unique closed $G_0$-orbit in $\bbP(V)$.   The stabilizer in the semisimple part $F = G_0^{\Ss}$ of the highest weight line $l_0 \in \cV \subset \bbP(V)$ is a parabolic subgroup $Q \subset F$, and this induces a $|1|$-grading $\ff = \ff_{-1} \op \ff_0 \op \ff_1$ with $\fq = \ff_0 \op \ff_1$.  Furthermore, $\cV \subset \bbP(V)$ is smooth, irreducible, and {\em Legendrian}, i.e.\ $\widehat{T}_l \cV \in \tLG(V)$ at any $l \in \cV$.  Here, the {\em affine tangent space} $\widehat{T}_l \cV \subset V$ is the span of $l$ and the tangent space to the cone over $\cV$ at any nonzero point along $l$.

 \begin{table}[h]
 \[
 \begin{array}{|cccccc|} \hline
 G / P & \mbox{Range} & F / Q & \begin{array}{c} V = \fg_{-1} \mbox{ as} \\ \mbox{an } \ff\mbox{-module} \end{array}& \dim(V) & \cV \subset \bbP(V)\\ \hline\hline
 B_\ell / P_2 & \begin{array}{c} \ell \geq 4\\ \ell=3 \end{array} & \begin{array}{l} A_1/P_1 \times B_{\ell-2}/P_1\\ A_1 / P_1 \times A_1 / P_1 \end{array} & \begin{array}{@{\quad\,}l} \bbC^2 \boxtimes \bbV_{\lambda_1}\\ \bbC^2 \boxtimes S^2\bbC^2 \end{array} & 2(2\ell-3) & \Seg(\bbP^1 \times \bbQ^{2\ell-5})\\  \hline
 D_\ell / P_2 & \ell \geq 5 & A_1/P_1 \times D_{\ell-2}/P_1 & \bbC^2 \boxtimes \bbV_{\lambda_1} & 2(2\ell-4) & \Seg(\bbP^1 \times \bbQ^{2\ell-6})\\ \hline
 G_2 / P_2 & - & A_1/P_1 & S^3 \bbC^2 & 4 & \mbox{twisted cubic } \nu_3(\bbP^1) \\
 D_4 / P_2 & - & (A_1/P_1)^3 & \bbC^2 \boxtimes \bbC^2 \boxtimes \bbC^2 & 8 & \Seg(\bbP^1 \times \bbP^1 \times \bbP^1) \\
 F_4 / P_1 & - & C_3/P_3 & \bbV_{\lambda_3} & 14 & \tLG(3,6)\\
 E_6 / P_2 & - & A_5/P_3 & \bbV_{\lambda_3} & 20 & \tGr(3,6)\\
 E_7 / P_1 & - & D_6/P_6 & \bbV_{\lambda_6} & 32 & \mbox{$D_6$-spinor variety} \\
 E_8 / P_8 & - & E_7/P_7 & \bbV_{\lambda_7} & 56 & \mbox{Freudenthal variety} \\ \hline
 \end{array}
 \]
 \caption{Sub-adjoint varieties}
 \label{F:subadjoint}
 \end{table}

 We can arrive at Table \ref{F:subadjoint} in a uniform manner via the Dynkin diagram $\fD(\fg)$:
 \begin{itemize}
 \item Given $P = P_j \subset G$, remove the contact node $j$ from $\fD(\fg)$ to obtain $\fD(\ff)$.
 \item For every node $i$ connected to $j$ in $\fD(\fg)$: inscribe a 1 over $i$ if the bond is simple or is directed from $i$ to $j$; otherwise inscribe the multiplicity of the bond.  This yields $V = \fg_{-1}$ as an $\ff$-module.
 \item Crossed nodes for $Q \subset F$ correspond to the neighbouring nodes $N(j)$ to $j$ in $\fD(\fg)$.
 \end{itemize}
 
 \begin{example} \label{EX:G2P2} \raisebox{-0.05in}{\includegraphics[width=2.5cm]{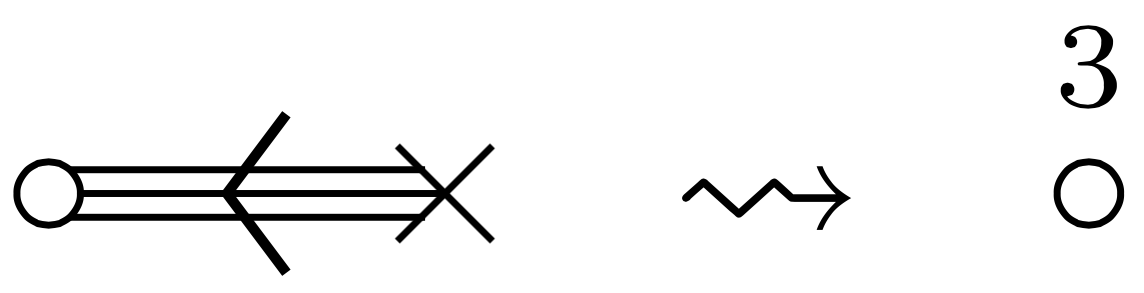}} \quad indicates that for $G_2 / P_2$, $V = \fg_{-1} \cong S^3 \bbC^2$ as an irrep of $A_1 \cong \fsl_2$.
 \end{example}
 
 We have several naturally associated objects inheriting $G_0$-invariance from  $\cV \subset \bbP(V)$:
 \begin{enumerate}
 \item   Let $\widehat\cV$ denote the image of the embedding $\cV \to \tLG(V)$ given by $l \mapsto \widehat{T}_l \cV$.

 \item Let $\widetilde\cV := \bigcup_{l \in \cV} \{ L \in \tLG(V) : l \subset L \} \subsetneq \operatorname{LG}(V)$.  (This is a hypersurface.)

 \item The tangential variety $\tau(\cV) = \bigcup_{l \in \cV} \bbP(\widehat{T}_l \cV) \subset \bbP(V)$ is a {\em quartic hypersurface}, so $\tau(\cV) = \{ \sfQ = 0 \}$ for some symmetric tensor $\sfQ \in S^4 V^*$.  Let $[\sfQ] = \{ c \sfQ : c\neq 0 \}$ denote its conformal class.
 \end{enumerate}
 
 \begin{example}[$G_2 / P_2$] \label{EX:G2-1} Here, $V = \fg_{-1} \cong S^3 \bbC^2$ as a module for $\fg_0 \cong \fgl_2$.   Let $\bbC^2 = \tspan\{ \sfr,\sfs \}$, so $\fgl_2$ is spanned by
 $\sfI = \sfr \partial_\sfr + \sfs \partial_\sfs,\,
 \sfE = \sfr \partial_\sfs, \,
 \sfH = \sfr \partial_\sfr - \sfs \partial_\sfs, \,
 \sfF = \sfs \partial_\sfr$.  Then $V$ has a $\tGL_2$-invariant CS-form $[\eta]$, where
 \begin{align} \label{E:G2-eta}
 \eta( f,g ) := \frac{1}{3!} (f_{\sfr\sfr\sfr} g_{\sfs\sfs\sfs} - 3 f_{\sfr\sfr\sfs} g_{\sfs\sfs\sfr} + 3 f_{\sfr\sfs\sfs} g_{\sfs\sfr\sfr} - f_{\sfs\sfs\sfs} g_{\sfr\sfr\sfr}).
 \end{align}
The twisted cubic $\cV = \{ [v^3] : [v] \in \bbP^1 \} \subset \bbP(V)$ is $\tGL_2$-invariant.  In $V$, differentiating $\gamma(t) = (\sfr+t\sfs)^3$ at $t=0$ yields the osculating sequence $V^0 \subset V^{-1} \subset V^{-2} \subset V^{-3} = V$, where $V^0 = \tspan\{ \sfr^3 \}$, $V^{-1} := \widehat{T}_{[x^3]} \cV = \tspan\{ \sfr^3, \sfr^2 \sfs \}$ is Legendrian, and $V^{-2} = \tspan\{ \sfr^3, \sfr^2\sfs, \sfr \sfs^2 \}$.
 In the dual basis $\theta^1, \theta^2, \theta^3, \theta^4$ to $(\sfr^3,3\sfr^2\sfs,3\sfr\sfs^2,\sfs^3)$, we have $\eta = 6(\theta^1 \wedge \theta^4 - 3 \theta^2 \wedge \theta^3)$.  The discriminant of $f = a_1 \sfr^3 + 3 a_2 \sfr^2 \sfs + 3 a_3 \sfr \sfs^2 + a_4 \sfs^3$ is:
 \begin{align} \label{E:G2-Q}
 \sfQ = (\theta^1)^2 (\theta^4)^2 - 6 \theta^1 \theta^2 \theta^3 \theta^4 + 4 \theta^1 (\theta^3)^3 + 4 (\theta^2)^3 \theta^4 - 3 (\theta^2)^2 (\theta^3)^2,
 \end{align}
 and this is conformally $G_0$-invariant. The locus $\sfQ = 0$ consists of all binary cubics with a multiple root.
 \end{example}
 
 When $G = D_4$, $\sfQ$ is Cayley's hyperdeterminant.
 
 \begin{lemma} If $G \neq A_\ell, C_\ell$, then $\ff = \fg_0^{\Ss} \subsetneq \fsp(V)$ is a maximal subalgebra.
 \end{lemma}
 
 \begin{proof}
 There are no proper $\ff$-invariant subspaces of $V$, so the inclusion $\ff \inj \fsp(V)$ is {\em irreducible}.   From Dynkin \cite{Dyn1957} (see also \cite[Chp.\ 6, Thms.\ 3.1--3.3]{OV1994}), the maximal subalgebras $\fm \inj \fsp(V)$ are:
 \begin{itemize}
 \item $\fm$ non-simple: $\fsp(V) = \fsp(\bbV_1 \otimes \bbV_2)$ and $\fm = \fsp(\bbV_1) \times \fso(\bbV_2)$, where $d_i = \dim(\bbV_i)$ satisfy $d_1 \geq 2$; $4 \neq d_2 \geq 3$ or $(d_1,d_2) = (2,4)$.  This is true for $\ff$ when $G = B_\ell$ or $D_\ell$ with $d_1 = 2$.
 \item $\fm$ simple: Aside from the exceptions in \cite[Table 7]{OV1994}, all non-trivial irreps $\psi : \fm \to \fsp(V)$ yield $\psi(\fm) \subset \fsp(V)$ a maximal subalgebra.  This is true for $\ff$ when $G$ is exceptional.
 \end{itemize}
 \end{proof}

 \begin{prop} \label{P:g0-reduction}
 Given the sub-adjoint variety $\cV \subset \bbP(V)$ for $G \neq A_\ell, C_\ell$, any of $\cV$, $\widehat\cV$, $\widetilde\cV$, or $[\sfQ]$ reduces the structure algebra $\fcsp(V)$ to $\fg_0$.
 \end{prop}
 
 \begin{proof} Let $\fs \subset \fcsp(V)$ be the Lie algebra of the stabilizer of any of the given objects, so $\fg_0 \subset \fs$.  We have $\fcsp(V) = \bbC \times \fsp(V)$ with $\bbC = \fz(\fg_0)$.  Since $\ff \subset \fsp(V)$ is maximal, the result follows.
 \end{proof}

 \subsection{Jordan algebras and sub-adjoint varieties}
 \label{S:LM}
 
 Sub-adjoint varieties $\cV \subset \bbP(V)$ admit a remarkably uniform description in terms of Jordan algebras, which we review here.
 
 Fixing $l_0 \in \cV$, we have $\cV \cong F/Q$.  Let $V^0 \subset V^{-1} \subset V^{-2} \subset ... \subset V = \fg_{-1}$ be the corresponding ($Q$-invariant) osculating sequence at $l_0$.  (In particular, $V^{-1} = \widehat{T}_{l_0} \cV$.)  In all our cases, $V^{-3} = V$.  This filtration has as its associated-graded $\tgr(V) = \bop_{i \leq 0} V_i$, where $V_i := V^i / V^{i+1}$, and this is naturally an $F_0$-module.  Since $\ff = \ff_{-1} \op \fq$ and $V^{-1} = \ff \cdot l_0$, then the (intrinsic) tangent space $T_{l_0} \cV \cong T_o (F/Q) \cong \ff_{-1} \cdot l_0$ is identified with $W:= V_{-1}$ as $\ff_0$-modules.

 In \cite[\S 5.1]{LM2001}, Landsberg and Manivel gave the following $\ff_0^{\Ss}$-module descriptions\footnote{In \cite[\S 5.1]{LM2001}, note that our $\ff,\ff_0$ are their $\fg,\fl$ respectively.  Also, while \cite{LM2001} mainly concentrates on the exceptional cases, the first sentence of \cite[p.496]{LM2001} indicates that the spin factor cases similarly satisfy \eqref{E:f-decomp}--\eqref{E:V-decomp}.
} of $\ff$ and $V$:
 \begin{align}
 \ff &= \ff_{-1} \op \ff_0 \op \ff_1 \cong W \op \ff_0 \op W^*, \label{E:f-decomp}\\
 V &\cong V_0 \op V_{-1} \op V_{-2} \op V_{-3} \cong \bbC \op W \op W^* \op \bbC, \label{E:V-decomp}
 \end{align}
 where $W$ is the (complex) Jordan algebra\footnote{The Jordan {\em algebra} structure will not play any explicit role in this article.  Instead, $\fC$ will play a fundamental role.}  corresponding to $G$ (Table \ref{F:sym}), which admits a natural cubic form $\fC \in S^3 W^*$ with symmetry algebra $\ff_0^{\Ss}$ (Table \ref{F:extended-magic}). Such $W$ and $\fC$ are given below:
 \begin{enumerate}
 \item[(i)] $3 \times 3$ $\bbA$-hermitian matrices $W = \cJ_3(\bbA)$, where $\bbA$ is a {\em complex} composition algebra, i.e.\ $\bbA = \bbA_\bbR \otimes_\bbR \bbC$ where $\bbA_\bbR$ is $\underline{0}$ (trivial algebra) or $\bbR, \bbC, \bbH, \bbO$.  Here, $\fC(t^3) = \det(t)$ is the determinant, defined via the Cayley--Hamilton identity (see \cite[eq (5.7)]{SV2000}): $ t^3 - \tr(t) t^2 - \frac{1}{2}( \tr(t^2) - \tr(t)^2) t - \det(t) \id = 0.$
 \item[(ii)] $W = \GJ := \bbC$ equipped with $\fC(t^3) = \frac{t^3}{3}$.
 \item[(iii)] Spin factor $W = \mathcal{J\!S}_m := \bbC^m \op \bbC$, $m \geq 1$, where $\bbC^m$ carries a non-degenerate symmetric bilinear form $\langle \cdot, \cdot \rangle$.  Here, $\fC(t^3) = \langle v,v \rangle \lambda$, where $t = (v,\lambda)$.
 We will often use an adapted basis: Let $\sfw_\infty = 1$ span the $\bbC$-factor and pick a basis $\{ \sfw_\sfa \}_{\sfa=1}^m$ of $\bbC^m$ with $\fC_{\infty \sfa\sfb} = \langle \sfw_\sfa, \sfw_\sfb \rangle = \delta_{\sfb,\sfa'}$, where $\sfa' := m+1 - \sfa$.  
 \end{enumerate}

 \begin{table}[h]
 \[
 \begin{array}{|c|cccccccc|} \hline
 &&& & & A_1 & A_2 & C_3 & F_4\\
 \ff_0^{\Ss} & B_{\ell-3} & D_{\ell-3} && & A_2 & A_2 \times A_2 & A_5 & E_6\\
 \ff = \fg_0^{\Ss} & A_1 \times B_{\ell-2} & A_1 \times D_{\ell-2} & A_1 & A_1 \times A_1 \times A_1 & C_3 & A_5 & D_6 & E_7\\
 \fg & B_\ell & D_\ell & G_2 & D_4 & F_4 & E_6 & E_7 & E_8\\ \hline
 \end{array}
 \]
 \caption{A magic rectangle}
 \label{F:extended-magic}
 \end{table}
 
 On $V \cong \tgr(V)$, we have the structure of a graded $\ff_0^{\Ss}$-algebra \cite[Cor.3.8]{LM2003} (induced from $\cV$ and the choice of $l_0$). The non-degenerate pairing $V_{-1} \times V_{-2} \to V_{-3} \cong \bbC$ then identifies $V_{-2} \cong W^*$, while $\fC$ arises from the (symmetric) pairing $V_{-1} \times V_{-1} \to V_{-2}$. The highest weight of $W \cong \ff_{-1}$ as a $\ff_0^{\Ss}$-module is obtained from $F/Q$ analogous to how the highest weight of $V = \fg_{-1}$ as $\fg_0^{\Ss}$-module was obtained from $G/P$.
 
 \begin{lemma} \label{L:C-inj}
 $\fC : W \to S^2 W^*$ is injective.
 \end{lemma}
 
 \begin{proof} This is immediate for the spin factor, $\GJ$, and $\cJ_3(\underline{0})$ cases.  For $W = \cJ_3(\bbA)$, $S^2 W^* = S^2_0 W^* \op W$ as a sum of $\ff_0^{\Ss}$-irreps ($S^2_0 W^* \subset S^2 W^*$ denotes the highest weight component), e.g.\ when $G = E_8$, we have $\ff_0^{\Ss} = E_6$, and $W,W^*,S^2_0 W^*$ have $\ff_0^{\Ss}$-weights $\lambda_6,\lambda_1, 2\lambda_1$.  The claim follows by Schur's lemma.
 \end{proof}
 
 Fix a basis $\{ \sfw_a \}_{a=1}^{n-1}$ of $W$ and $\{ \sfw^a \}_{a=1}^{n-1}$ its dual basis.  On $V = \bbC \op W \op \bbC \op W^*$, take the basis 
 \begin{align} \label{E:CS-basis}
 \sfb_0 = 1, \quad \sfb_a = -6\sfw_a, \quad \sfb^0 = -4, \quad \sfb^a = -4\sfw^a.
 \end{align}
 \begin{framed}
  Notation: Given $t = t^a \sfw_a \in W$, write $\fC(t^3) := \fC(t,t,t) = \fC_{abc} t^a t^b t^c \in \bbC$, while $\fC_a(t^2) := \fC_{abc} t^b t^c = \frac{1}{3} \partial_{t^a}( \fC(t^3) )$ and $\fC_{ab}(t) := \fC_{abc} t^c = \frac{1}{6} \partial_{t^a} \partial_{t^b}( \fC(t^3) )$.
 \end{framed}
 
 We have the following descriptions of $\cV, \widehat\cV, \widetilde\cV$, and $[\sfQ]$, which are derived from Landsberg--Manivel \cite{LM2001}.  
 
 \begin{prop} The basis \eqref{E:CS-basis} is a CS-basis on $V$ for the ($\ff$-invariant) symplectic form given in \cite[Prop.5.4]{LM2001}.  In this basis, $\cV$ is locally parametrized about $l_0 = [1,0,0,0]$ by $t = t^a \sfw_a \in W$ via
 \begin{align} \label{E:SA-var}
 \phi: t \mapsto \left[1, -t^a, -\frac{1}{2} \fC(t^3), -\frac{3}{2} \fC_a(t^2)\right].
 \end{align}
 In standard coordinates $(u_{ij})$ (see \S \ref{S:LG}) about $o = \bbC \op W \in \tLG(V)$ induced from \eqref{E:CS-basis}, we locally have
 \begin{align}
 \widehat\cV: &\quad (u_{ij}) = \displaystyle\left( \begin{array}{c|c} u_{00} & u_{0b} \\ \hline u_{a0} & u_{ab}\end{array} \right) = \left( \begin{array}{c|c} \fC(t^3) & \frac{3}{2} \fC_b(t^2)\\[0.05in] \hline \frac{3}{2} \fC_a(t^2) & 3\fC_{ab}(t)  \end{array} \right); \label{E:whV}\\
 \widetilde\cV: &\quad \left\{ \begin{array}{l} u_{00} = t^a t^b u_{ab} - 2\fC(t^3)\\ u_{a0} = t^b u_{ab} - \frac{3}{2} \fC_a(t^2)\end{array} \right.. \label{E:wtV}
 \end{align}
 In particular, $\dim(\widehat\cV) = \dim(W)$ and $\codim(\widetilde\cV) = 1$. 
 \end{prop}
 
 \begin{proof} The first claim is clear and \eqref{E:SA-var} follows from $\phi$ in \cite[Sec.1.2]{LM2001}.  Put the components of $l = \phi(t)$ and $\frac{\partial \phi}{\partial t^b}$ into the rows of a matrix and then row reduce to obtain $\widehat\cV$:
 \begin{align} \label{E:Vhat-mat}
 \left( \begin{array}{cc|cc}
 1 & -t^a & -\frac{\fC(t^3)}{2} & -\frac{3\fC_a(t^2)}{2}\\
 0 & -\delta_b{}^a & -\frac{3\fC_b(t^2)}{2} & -3\fC_{ba}(t) 
 \end{array} \right) \leadsto
 \left( \begin{array}{cc|cc}
 1 & 0 & \fC(t^3) & \frac{3\fC_a(t^2)}{2}\\
 0 & \delta_b{}^a & \frac{3\fC_b(t^2)}{2} & 3\fC_{ba}(t) 
 \end{array} \right).
 \end{align}
 Now for $\widetilde\cV$, let $L \in \tLG(V)$ have standard coordinates $(u_{ij})$.  Then row reduce
 \[
 \left( \begin{array}{cc|cc}
 1 & 0 & u_{00} & u_{0a}\\
 0 & \delta_b{}^a & u_{b0} & u_{ba}\\
 1 & -t^a & -\frac{\fC(t^3)}{2} & -\frac{3\fC_a(t^2)}{2}
 \end{array} \right) \leadsto
 \left( \begin{array}{cc|cc}
 1 & 0 & u_{00} & u_{0a}\\
 0 & \delta_b{}^a & u_{b0} & u_{ba}\\
 0 & 0 & -u_{00} + t^b u_{b0} -\frac{\fC(t^3)}{2} & -u_{0a} + t^b u_{ba} -\frac{3\fC_a(t^2)}{2}
 \end{array} \right).
 \]
 For the incidence condition $l \subset L$, the bottom row must be zero, and this yields $\widetilde\cV$. 
 \end{proof} 
 
 Remarkably, $\widehat\cV$ and $\widetilde\cV$ can also be derived via an envelope construction:
 
 \begin{cor} \label{R:Goursat} Consider the family of hypersurfaces $\fG_t = u_{00} - 2 t^a u_{a0} + t^a t^b u_{ab} - \fC(t^3) = 0$ in $\tLG(V)$ parametrized by $t \in W$.  Its first and second order envelopes are $\widetilde\cV$ and $\widehat\cV$ respectively.
 \end{cor}
 
 \begin{proof}
 We readily verify that $\widetilde\cV = \{ \fG_t = 0, \frac{\partial \fG_t}{\partial t^a} = 0 \}_{t \in W}$ and $\widehat\cV = \{ \fG_t = 0, \frac{\partial \fG_t}{\partial t^a} = 0, \frac{\partial^2 \fG_t}{\partial t^a \partial t^b} = 0  \}_{t \in W}$.
 \end{proof}

 To describe $\sfQ$, we use the {\em dual cubic} $\fC^* \in S^3 W$ (see Table \ref{F:dual-cubic}), induced from $\fC \in S^3 W^*$ via a multiple of the trace form $t \mapsto \tr(t^2)$ on the Jordan algebra $W$.  To fix this multiple, we use the normalization\footnote{See \cite[Lemma 5.2.1(iv)]{SV2000} or \cite[\S 4.3]{Man2013} for why the identity \eqref{E:dual-cubic-id} (up to scale) exists.}
 \begin{align} \label{E:dual-cubic-id}
 \fC^*(\fC(t^2)^2) = \frac{4}{27} \fC(t^3) t, \quad\qbox{i.e.} (\fC^*)^{abc} \fC_{bde} \fC_{cfg} = \frac{4}{27} \fC_{(def} \delta_{g)}{}^a.
 \end{align}
 (Rescaling $\fC$ by $\lambda$ forces $\fC^*$ to rescale by $\frac{1}{\lambda}$.)  Note that $\fC^*(\fC(t^2)^3) = \frac{4}{27} \fC(t^3)^2$.  But setting $s^* = \fC(r^2)$ and $t=r$ in the equation preceding \cite[Lemma 5.6]{LM2001} yields $\fC^*(\fC(r^2)^3) = 2\fC(r^3)^2$, so our $\fC^*$ is $\frac{2}{27}$ times theirs.  

 \begin{table}[h]
 \[
 \begin{array}{|c|c|c|c|c|} \hline
 W & \mbox{Trace form on $W$} & \mbox{Induced $\sharp : W^* \to W$} & \fC(t^3) & \fC^*((s^*)^3)\\ \hline\hline
 \cJ_3(\bbA) & t \mapsto \tr(t^2) & 
 \begin{psmallmatrix}
 \lambda_1 & v_1 & v_2\\
 \overline{v_1} & \lambda_2 & v_3\\
 \overline{v_2} & \overline{v_3} & \lambda_3
 \end{psmallmatrix} \mapsto 
 \begin{psmallmatrix}
 \lambda_1 & \frac{1}{2} v_1 & \frac{1}{2} v_2\\
 \frac{1}{2} \overline{v_1} & \lambda_2 & \frac{1}{2} v_3\\
 \frac{1}{2} \overline{v_2} & \frac{1}{2} \overline{v_3} & \lambda_3
 \end{psmallmatrix} & \det(t) & 4\det((s^*)^\#)\\ \hline
 \GJ & t \mapsto t^2 & t \mapsto t & \frac{t^3}{3} & \frac{4(s^*)^3}{9}\\ \hline
 \JS_m & t = (v,\lambda) \mapsto \langle v,v \rangle + \lambda^2 & (v^*,\mu) \mapsto ((v^*)^{\sharp},\mu) & \langle v,v \rangle \lambda & \langle (v^*)^{\sharp}, (v^*)^{\sharp} \rangle \mu\\ \hline
 \end{array}
 \]
 \caption{Dual cubic $\fC^* \in S^3 W$ in our normalization \eqref{E:dual-cubic-id}}
 \label{F:dual-cubic}
 \end{table}
 
 \begin{prop}
 Let $(\alpha,r_a,\beta^*,s^a)$ be coordinates wrt \eqref{E:CS-basis} and let $r = r^a\sfw_a$ and $s^* = s_a \sfw^a$.  Then
 \begin{align} \label{E:quartic}
 \sfQ(\alpha,r^a,\beta^*,s_a) = (\alpha\beta^* + \langle r, s^* \rangle )^2 + 2 \beta^* \fC(r^3) - 2\alpha \fC^*((s^*)^3) - 9 \langle \fC(r^2), \fC^*((s^*)^2) \rangle.
 \end{align}
 \end{prop}
 
 \begin{proof}
 This follows from \cite[Prop.5.5]{LM2001} by replacing their $\fC^*$ by $\frac{27}{2} \fC^*$ to get to our normalization.  Then substitute $v = (\alpha,-6r, -4s^*,-4\beta^*)$ and rescale $\sfQ$.
 \end{proof}
 
 \begin{example}[$G_2 / P_2$]  Continuing Example \ref{EX:G2-1}, let $\fC(t^3) = \frac{t^3}{3}$.  In the CS-basis $(\sfb_0,\sfb_1,\sfb^0,\sfb^1) = (\sfr^3,-3\sfr^2\sfs,-6\sfs^3,-6\sfr\sfs^2)$, $[(\sfr + t\sfs)^3]$ takes the form in \eqref{E:SA-var} and $\sfQ$ from \eqref{E:G2-Q} takes the form in \eqref{E:quartic}.
 
 On $\bigwedge^2 V$, define $(\cdot,\cdot)$ by $ (f_1 \wedge f_2, g_1 \wedge g_2) = \vol_V(f_1,f_2,g_1,g_2)$, where $0 \neq \vol_V \in \bigwedge^4 V^*$.  The Pl\"ucker embedding identifies $\tLG(2,4)$ with $\cQ = \{ [z] : (z,z) = 0 \} \subset \bbP(\bigwedge^2_0 V) \cong \bbP^4$, where elements of $\bigwedge^2_0 V$ contract trivially with $\eta$.  Then $\widehat\cV \subset \cQ \subset \bbP^4$ and $\tau(\widehat\cV) = \widetilde\cV$.  (The latter does not hold for other $G$.)
 
  About $o = \tspan\{ \sfb_0, \sfb_1 \}$, standard coordinates $(u_{00}, u_{01}, u_{11})$ on $\tLG(2,4)$ correspond to $\tspan\{ \sfb_0 + u_{00} \sfb^0 + u_{01} \sfb^1, \sfb_1 + u_{01} \sfb^0 + u_{11} \sfb^1 \}$.  Via Pl\"ucker, this is $(1,u_{00},u_{01},u_{11},u_{00} u_{11} - (u_{01})^2)$ with respect to the basis $\sfb_0 \wedge \sfb_1, \, \sfb^0 \wedge \sfb_1, \, \sfb_0 \wedge \sfb^0 - \sfb_1 \wedge \sfb^1, \, \sfb_0 \wedge \sfb^1, \, \sfb^0 \wedge \sfb^1$ on $\bigwedge^2_0 V$.  Using \eqref{E:whV}, $\widehat\cV$ is a twisted quartic given by 
$\gamma(t) = \left(1,\frac{t^3}{3},\frac{t^2}{2},t,\frac{t^4}{12}\right)$, and $\gamma(t) + \mu\gamma'(t)$ is given by $u_{00} = \frac{t^3}{3} + \mu t^2, \, u_{01} = \frac{t^2}{2} + \mu t, \, u_{11} = t + \mu$ is its tangent developable.\footnote{In \cite{Car1893}, Cartan only briefly alluded to the Goursat parabolic PDE as the tangent developable for which the involutive system is the singular variety.  See \cite[p.161 -- eq.(7)]{Car1910} for the explicit model, which should read: $r \,{\color{red} +}\, x_5 s - \frac{1}{6} x_5^3 = 0,\,\, s + x_5 t + \frac{1}{2} x_5^2 = 0$.
}  Eliminating $\mu$ yields \eqref{E:wtV} for $\widetilde\cV$.  Note that $\widehat\cV$ is a null curve and $\widetilde\cV$ is a null surface in $\tLG(2,4)$ for the conformal structure $[du_{00} du_{11} - (du_{01})^2]$.
 \end{example}

 \section{Parabolic contact structures and flat models}
  
 \subsection{Contact geometry}
 \label{S:contact}
 
 We now summarize the geometric construction of jet spaces \cite{Yam1982,Olv1995,KLR2007}.
 
 Given a contact manifold $(M^{2n+1},\cC)$, the corank one contact distribution $\cC \subset \Gamma(TM)$ is completely non-integrable.  For any local defining 1-form $\sigma$ (unique up to a conformal factor), this means that $\sigma \wedge (d\sigma)^n \neq 0$ everywhere and so $\eta = (d\sigma)|_\cC$ yields a CS-form on $\cC$.  Define the Lagrange--Grassmann bundle $\pi : M^{(1)} \to M$ by letting $\tLG(\cC_m)$ be its fibre over $m \in M$.  Any $m^{(1)} \in M^{(1)}$ such that $\pi(m^{(1)}) = m$ corresponds to a Lagrangian subspace $L_{m^{(1)}} \subset \cC_m$, so this tautologically defines the {\em canonical distribution} $\cC^{(1)} \subset \Gamma(TM^{(1)})$ via $\cC^{(1)}_{m^{(1)}} = (\pi_*)^{-1}(L_{m^{(1)}})$.  (For higher-order prolongations $M^{(k)}$, see e.g.\ \cite{Yam1982}.)
 
 By Pfaff's theorem, there are local coordinates $(x^i,u,u_i)$ on $M$ such that $\sigma = du - u_i dx^i$, i.e.\ locally, $M$ is the first jet space $J^1(\bbC^n,\bbC)$.  With respect to $\eta = d\sigma = dx^i \wedge du_i$, $\cC$ has {\em standard CS-framing}
 \begin{align} \label{E:std-CS}
 \partial_{x^i} + u_i \partial_u, \quad \partial_{u_i}.
 \end{align}
 On $M^{(1)}$, take $\pi$-adapted coordinates $(x^i,u,u_i,u_{ij})$: about $o = \tspan\{ \partial_{x^i} + u_i \partial_u \}$, let fibre coordinates $u_{ij} = u_{ji}$ correspond to the Lagrangian subspace $\tspan\{ \partial_{x^i} + u_i \partial_u + u_{ij} \partial_{u_j} \}$ so that $\cC^{(1)}$ is given by
 \begin{align} \label{E:J2-cansys}
 \tspan\{ \partial_{x^i} + u_i \partial_u + u_{ij} \partial_{u_j}, \,\, \partial_{u_{ij}} \} = \ker\{ du - u_i dx^i, du_i - u_{ij} dx^j \}.
 \end{align}
 Locally, $M^{(1)}$ is the second jet space $J^2(\bbC^n,\bbC)$.

 Given a distribution $\cD$ on a manifold $N$, we may form its weak derived flag $\cD =: \cD^{-1} \subset \cD^{-2} \subset ...$.  Its associated-graded $\fg_-(n) = \cD(n) \op (\cD^{-2}(n) / \cD^{-1}(n)) \op ...$ at $n \in N$ is the {\em symbol algebra}.  This is a nilpotent graded Lie algebra, whose (tensorial) bracket is induced from the Lie bracket of vector fields on $N$.  The symbol algebras for $(M,\cC)$ and $(M^{(1)},\cC^{(1)})$ are respectively modelled on:
 \begin{align}
 \fg_- &= \fg_{-1} \op \fg_{-2} \cong V \op \bbC, \label{E:J1-symbol}\\
 \fg_- &= \fg_{-1} \op \fg_{-2} \op \fg_{-3} \cong (L \op S^2 L^*) \op L^* \op \bbC, \label{E:J2-symbol}
 \end{align}
 where $\dim(V) = 2n$ and $\dim(L) = n$.  The former is the Heisenberg Lie algebra, while the non-trivial brackets on the latter are all natural contractions.  We note that $S^2 L^*$ corresponds to a distinguished subbundle of $\cC^{(1)}$, namely the vertical bundle for $\pi : M^{(1)} \to M$.
   
 A contact transformation of $(M,\cC)$ is a diffeomorphism $\phi : M \to M$ such that $\phi_*(\cC) = \cC$.  Infinitesimally, $\bX \in \Gamma(TM)$ is contact if $\cL_\bX \cC \subset \cC$.  These definitions apply similarly for $(M^{(1)},\cC^{(1)})$, but more can be said: by B\"acklund's theorem, any contact transformation [vector field] of $(M^{(1)},\cC^{(1)})$ is the {\em prolongation} of one on $(M,\cC)$.  (See \cite{Olv1995} for the standard prolongation formula yielding $\bX^{(1)} \in \Gamma(TM^{(1)})$ from $\bX$.) 
 
   On $J^1(\bbC^n,\bbC)$, any contact vector field is uniquely determined by a function on $M$ called its {\em generating function}.  Conversely, any $f = f(x^i,u,u_i)$ is a generating function for a contact vector field via
 \begin{align} \label{E:Xf}
 \bS_f = -f_{u_i} \partial_{x^i} + (f - u_i f_{u_i}) \partial_u + (f_{x^i} + u_i f_u) \partial_{u_i} 
 = -f_{u_i} \frac{d}{dx^i} + f \partial_u + \frac{df}{dx^i} \partial_{u_i},
 \end{align}
 where $\frac{d}{dx^i} := \partial_{x^i} + u_i \partial_u$. If $g$ is another generating function, the commutator $[\bS_f, \bS_g]$ is a contact vector field $\bS_{[f,g]}$, where the {\em Lagrange bracket} $[f,g]$ is given by
 \begin{align} \label{E:Lagrange}
 [f,g] = f g_u - g f_u + \frac{df}{dx^i} g_{u_i} - \frac{dg}{dx^i} f_{u_i}.
 \end{align}

A (system of) second order PDE in one dependent variable and $n$-independent variables corresponds to a submanifold $\cR \subset \tLG(\cC) = M^{(1)}$ transverse to $\pi$.  The distribution $\cC^{(1)}$ and its derived system $(\cC^{(1)})^{-2}$ induce distributions $\cD$ and $\widetilde\cC$ on $\cR$, and $(\cR;\cD,\widetilde\cC)$ is called a {\em PD-manifold} \cite{Yam1982}.  By \cite[Thm.4.1]{Yam1999}, all symmetries of $(\cR;\cD,\widetilde\cC)$ correspond to (external) contact symmetries of $\cR \subset M^{(1)}$, i.e.\ contact transformations of $M^{(1)}$ preserving $\cR$.   Define $\cR^{(1)}$ as the collection of $n$-dimensional integral elements for $(\cR,\cD)$ transverse to $\pi : M^{(1)} \to M$.  From \cite[Thm.4.2]{Yam1999}, \cite[Prop.5.11]{Yam1982}, if $\cR^{(1)} \to \cR$ is surjective, then for any $v \in \cR$,
 \begin{align} \label{E:PD-Cauchy}
 \dim(\widetilde\cC(v)) - \dim(\cD^{-2}(v)) = \dim(\Ch(\cD)(v)),
 \end{align}
 where $\Ch(\cD) = \{ \bX \in \Gamma(\cD) : \cL_\bX \cD \subset \cD \}$ is the {\em Cauchy characteristic} space of $\cD$.

 %%%%%%%%%%%%%%%%%%%%%%%%%%%%%%%%%%%%%%%%%%
 
 \subsection{$G$-contact structures}
 \label{S:G-contact}
 
 Given a contact manifold $(M,\cC)$ with symbol algebra $\fg_-(m)$ at $m \in M$ modelled on the Heisenberg algebra $\fg_-$, the graded frame bundle $F_{\tgr}(M) \to M$ has fibre over $m \in M$ consisting of all graded Lie algebra isomorphisms $\iota : \fg_- \to \fg_-(m)$.  Its structure group is $\CSp(\fg_{-1})$.  
   
 \begin{defn} \label{D:G-contact} Let $G \neq A_1,C_\ell$ be a complex simple Lie group and $G^{\ad} \cong G/P$.  Let $G_0 \subset P$ be the reductive part.  A {\em $G$-contact structure} is a contact manifold $(M,\cC)$ of dimension $\dim(G/P)$ whose graded frame bundle $F_{\tgr}(M) \to M$ has structure group reduced according to the homomorphism $G_0 \to \CSp(\fg_{-1})$.
 \end{defn}
 
 A (local) equivalence of $G$-contact structures is a (local) contact transformation whose pushforward preserves the graded frame bundle reductions.  The fundamental theorem of Tanaka, Morimoto, and \v{C}ap--Schichl (see \cite{CS2009} for definitions and references) establishes an equivalence of categories between $G$-contact structures and (regular, normal) parabolic geometries of type $(G,P)$. Well-known consequences \cite{CS2009} are:
 \begin{itemize}
 \item Any such structure has symmetry dimension at most $\dim(\fg)$.
 \item There is a unique local model (the ``flat model'') with maximal symmetry dimension $\dim(\fg)$ and this has symmetry algebra isomorphic to $\fg$.
 \item $G$-contact structures are all non-rigid geometries, i.e.\ there exist non-flat models.
 \end{itemize}
 In spite of these general results arising from the broader theory of parabolic geometries, concrete local descriptions of $G$-contact structures have been lacking in the literature.  Indeed, we only know of Engel's twisted cubic model \cite{Eng1893} and the (contact) conformal quartic description \cite{Nur2013,LNS2016}.
 
Restrict now to $G \neq A_\ell,C_\ell$. Since $\fg_0 \subsetneq \fcsp(\fg_{-1})$ is a {\em maximal} subalgebra (Proposition \ref{P:g0-reduction}), the required structure group reduction (up to possibly a discrete subgroup) is mediated by a {\em field} of sub-adjoint varieties $\cV$ {\em or} any of $\widehat\cV,\widetilde\cV, [\sfQ]$, e.g.\ we require any graded isomorphism $\iota : \fg_- \to \fg_-(m)$ to map the model $\cV \subset \bbP(\fg_{-1})$ projectively onto $\cV_m \subset \bbP(\cC_m)$.  In \S \ref{S:LM}, these were given in a CS-basis, so a (local) $G$-contact structure is determined by a (local) CS-framing $\{ \bX_i, \bU^i \}_{i=0}^{n-1}$ on $\cC$ or its dual coframing $\{ \omega^i, \theta_i \}_{i=0}^{n-1}$.  The former induces fibre coordinates $p_{ij} = p_{ji}$ on $M^{(1)}$ corresponding to the Lagrangian subspace $\tspan\{ \bX_i + p_{ij} \bU^j \} = \ker\{ \theta_i - p_{ij} \omega^j \}$.  Thus, a $G$-contact structure is equivalently any of:
 \begin{itemize}
 \item a field of sub-adjoint varieties $\cV \subset \bbP(\cC)$, given by the projectivization of the vector fields
 \begin{align} \label{E:G-SA-var}
 \bV(\lambda,t) = \lambda^3\bX_0 - \lambda^2 t^a \bX_a - \frac{1}{2} \fC(t^3) \bU^0 - \frac{3}{2} \lambda\fC_a(t^2) \bU^a, \qquad \forall [\lambda,t] \in \bbP(\bbC \op W).
 \end{align}
 %%%%%%%%
 \item a field of tangential varieties $\tau(\cV) = \{ \sfQ = 0 \} \subset \bbP(\cC)$.  Letting $\Omega = \omega^a \otimes \sfw_a$ and $\Theta = \theta_a \otimes \sfw^a$, the (conformal) quartic $\sfQ \in S^4 \cC^*$ is given by
 \begin{align} \label{E:G-quartic}
  \sfQ = (\omega^i \theta_i )^2 + 2 \theta_0 \fC(\Omega^3) - 2 \omega^0 \fC^*(\Theta^3) - 9 \fC_a(\Omega^2) (\fC^*)^a(\Theta^2).
 \end{align}
 %%%%%%%%
 \item a 2nd order PDE (system) $\cE := \widehat\cV \subset \tLG(\cC) = M^{(1)}$, given in the CS-framing $\{ \bX_i, \bU^i \}$ by
 \begin{align} \label{E:G-involutive}
 (p_{ij}) 
 = \displaystyle\left( \begin{array}{c|c} p_{00} & p_{0b} \\ \hline p_{a0} & p_{ab}\end{array} \right) 
 = \left( \begin{array}{c|c} \fC(t^3) & \frac{3}{2} \fC_b(t^2)\\[0.05in] \hline \frac{3}{2} \fC_a(t^2) & 3\fC_{ab}(t)  \end{array} \right).
 \end{align}
 %%%%%%%%
 \item a single 2nd order PDE $\cF := \widetilde\cV \subset \tLG(\cC) = M^{(1)}$, given in the CS-framing $\{ \bX_i, \bU^i \}$ by
 \begin{align} \label{E:G-Goursat}
  \left\{ \begin{array}{l} p_{00} = t^a t^b p_{ab} - 2\fC(t^3)\\ p_{a0} = t^b p_{ab} - \frac{3}{2} \fC_a(t^2)\end{array} \right. .
 \end{align}
 \end{itemize}

 Now $(M_1,\cC_1,\cV_1)$ and $(M_2,\cC_2,\cV_2)$ are (locally) equivalent if there is a (local) contact transformation $\phi : M_1 \to M_2$ such that $\phi_*(\cV_1) = \cV_2$.   A symmetry is a self-equivalence of $(M,\cC,\cV)$.  A similar formulation holds for $[\sfQ]$.  For $\cE \subset M^{(1)}$, a symmetry is a contact transformation $\Phi : M^{(1)} \to M^{(1)}$ such that $\Phi(\cE) = \cE$, i.e.\ {\em external} symmetries.  By B\"acklund's theorem, $\Phi = \phi_*$ for some contact transformation $\phi : M \to M$.  Thus, symmetries of $\cE$ regarded as a field $m \mapsto \cE_m = \widehat\cV_m \subset \tLG(\cC_m)$ on $M$ are in 1-1 correspondence with external symmetries of $\cE \subset M^{(1)}$ regarded as a submanifold (PDE).  A similar formulation holds for $\cF$.  
  
If $\bS \in \Gamma(TM)$ is a contact vector field with prolongation $\bS^{(1)} \in \Gamma(TM^{(1)})$, then the infinitesimal symmetry condition for each of $\cV$, $[\sfQ]$, $\cE$, $\cF$ is correspondingly:
 \begin{align}
 \cL_\bS (\bV(\lambda,t)) \in \widehat{T}_{[\bV(\lambda,t)]} \cV; \qquad
 \cL_\bS \sfQ = \mu \sfQ \,\,\mbox{on}\,\, \cC; \qquad
 \cL_{\bS^{(1)}} \cE = 0 \,\,\mbox{on}\,\, \cE; \qquad
 \cL_{\bS^{(1)}} \cF = 0 \,\,\mbox{on}\,\, \cF.
 \end{align}
 
 \begin{prop} \label{P:same-sym}
 The symmetry algebra of $(M,\cC)$ endowed with any of $\cV$, or any of the induced fields $\cE = \widehat\cV$, $\cF = \widetilde\cV$, or $\tau(\cV) = \{ \sfQ = 0 \}$ is the {\em same}.
 \end{prop}
 
 \begin{proof}
 By Proposition \ref{P:g0-reduction}, each structure reduces the structure algebra of $F_{\tgr}(M) \to M$ according to the homomorphism $\fg_0 \to \fcsp(\fg_{-1})$ and these reductions are compatible since they are all induced from $\cV$.  At the group level, the reductions could potentially differ, but only by the action of a discrete group, which does not affect the (infinitesimal) symmetry algebra.
 \end{proof}
 
 This simple observation {\em dramatically} simplifies the (contact) symmetry computation for the PDE $\cE$ or $\cF$.  In particular, we avoid the complicated prolongation formula that yields $\bS^{(1)}$ from $\bS$ and instead we can equivalently find symmetries of $\cV$ or $[\sfQ]$ on $M$ itself.
 
 %%%%%%%%%%%%%%%%%%%%%%%%%%%%%%%%%%%%%%%%%%

 \subsection{Harmonic curvature and the flat $G$-contact structure}
 \label{S:Kh}
 
 A fundamental tensorial invariant for all (regular, normal) parabolic geometries is {\em harmonic curvature} $\kappa_H$.  It is a {\em complete} obstruction to flatness of the geometry.  Given the $G_0$-reduction $\cG_0 \subset F_{\tgr}(M)$ for a $G$-contact structure, $\kappa_H$ is a $G_0$-equivariant function valued in a cohomology space $H^2_+(\fg_-,\fg)$, or equivalently it is a section of the associated vector bundle $\cG_0 \times_{G_0} H^2_+(\fg_-,\fg)$ over $M$.  Concretely \cite[Chp.5]{CS2009}, for $G$-contact structures we find $\kappa_H$ as follows:
 \begin{enumerate}
 \item[(i)] Given any CS-framing $\{ \bX_i, \bU^i \}$ of $\cC$, define a partial connection $\nabla : \Gamma(TM) \times \Gamma(\cC) \to \Gamma(\cC)$ for which all frame vector fields are {\em parallel}.  Then $\nabla (\bV(\lambda,t)) = 0$ for any $[\lambda,t] \in \bbP(\bbC\op W)$, so the $G_0$-structure reduction is preserved.  Writing $[\bX_i, \bU^j] = \delta^i{}_j \bT \mod \cC$, we have $\bT \mod \cC \in \Gamma(TM/\cC)$ parallel for the induced connection on $TM/\cC$.
 \item[(ii)] Let $\bigwedge^2_0 \cC$ be kernel of the map $\bigwedge^2 \cC \to TM/\cC$ induced from the Lie bracket. The torsion $T^\nabla(X,Y) = \nabla_X Y - \nabla_Y X - [X,Y]$ restricts to a map $T^\nabla : \Gamma(\bigwedge^2_0 \cC) \to \Gamma(\cC)$.  In particular, {\em its components in the  CS-framing $\{ \bX_i, \bU^i \}$ above involve only the Lie bracket}.
 \item[(iii)] $\kappa_H$ is obtained by projecting $T^\nabla \in  \Gamma(\bigwedge^2_0 \cC^* \otimes \cC)$ to certain $\fg_0$-irreducible components obtained from a standard application of Kostant's theorem \cite{Kos1961,CS2009}.  (See \S\ref{S:G2P2} for the $G_2$-contact case.)
 \end{enumerate}
  
 This information is sufficient to identify the flat model for $G$-contact structures:

 \begin{theorem} \label{T:flat} Let $G \neq A_\ell,C_\ell$ be a complex simple Lie group.  Consider $(M,\cC) = (J^1(\bbC^n,\bbC),\cC)$ with standard jet space coordinates $(x^i,u,u_i)$, $0 \leq i \leq n-1$, standard CS-framing $\{ \bX_i = \partial_{x^i} + u_i \partial_u, \bU^i = \partial_{u_i} \}$ on $\cC$, and dual coframing $\{ \omega^i = dx^i, \theta_i = du_i \}$. Any of the models in Table \ref{F:ExcSimp} equivalently describes the flat $G$-contact structure.  This has (contact) symmetry algebra isomorphic to $\fg$.
 \end{theorem}

 \begin{proof} For the given CS-framing, $p_{ij} = u_{ij}$ in
 \eqref{E:G-SA-var}--\eqref{E:G-Goursat}, so we obtain Table \ref{F:ExcSimp}.
 The only non-trivial brackets among $\{ \bX_i, \bU^i \}$ are $[\bX_i,\bU^j] = \delta_i{}^j \partial_u$, so
 $\bigwedge^2_0 \cC = \tspan\{ \bX_i \wedge \bX_j, \bU^i \wedge \bU^j$ and $\bX_i \wedge \bU^j - \delta_i{}^j (\bX_0 \wedge \bU^0) \}$.  Pick $\nabla$ for which the CS-framing is parallel.  Then $T^\nabla|_{\bigwedge^2_0\cC} = 0$, so $\kappa_H = 0$ and the model is flat.
 \end{proof}
 
 \subsection{Symmetries of the flat $G$-contact structure}
  
Computing all symmetries of the flat $G$-contact structure via the PDE $\cE$ or $\cF$ is in general a hopeless task, but we will efficiently compute them via $\cV$ (see Table \ref{F:ExcSimp}).  We have $[\bS_f, \bV(\lambda,t)] \subset \widehat{T}_{\bV(\lambda,t)} \cV$, where $f$ is a generating function for a contact symmetry (see \eqref{E:Xf}).  The space of all such is equipped with the Lagrange bracket \eqref{E:Lagrange}.
 
 Clearly, $1, x^i, u_i$, and $\sfZ = 2u - x^i u_i$ are all (generating functions for) symmetries.  Indeed, $\bS_1 = \partial_u$, $\bS_{x^i} = \partial_{u_i} + x^i \partial_u$, and $\bS_{u_i} = -\partial_{x^i}$ all commute with $\bV(\lambda,t)$, while $\bS_\sfZ = x^i \partial_{x^i} + 2 u\partial_u + u_i \partial_{u_i}$ satisfies $[\bS_\sfZ,\bV(\lambda,t)] = \bV(\lambda,t)$.  (Alternatively, their prolongations to $J^2$ act trivially on all $u_{ij}$, while the defining equations for $\cE$ and $\cF$ only involve $u_{ij}$.)  Recall from \S \ref{S:SA-var} that $(G,P)$ induces a contact grading $\fg = \fg_{-2} \op \fg_{-1} \op \fg_0 \op \fg_1 \op \fg_2$.  At $o = \{ x^i = u = u_i = 0\}$, $\cC$ is spanned by all $\bS_{x^i}$ and $\bS_{u_i}$, so these correspond to $\fg_{-1}$.  Also, $[x^i,u_j] = \delta^i{}_j$, so $1 \in \fg_{-2}$.  Since $\sfZ$ acts by $-1$ on $\fg_{-1}$ and $-2$ on $\fg_{-2}$, this will serve as our grading element.  Because $\fg_-$ is known, and the brackets $\fg_{-1} \times \fg_2 \to \fg_1$ and $\fg_{-1} \times \fg_1 \to \fg_0$ are surjective, it suffices to determine $\fg_2$, which is only 1-dimensional.
  
 \begin{theorem} \label{T:sym-flat} Let $G \neq A_\ell, C_\ell$ be complex simple Lie group.  The flat $G$-contact structure from Theorem \ref{T:flat} admits the following symmetry that spans the top slot $\fg_2$ of the contact grading $\fg = \fg_{-2} \op ... \op \fg_2$:
\begin{align} \label{E:top-slot}
 f = u(u - x^i u_i) - \frac{1}{2} \fC(X^3) u_0 + \frac{1}{2} \fC^*(P^3) x^0 + \frac{9}{4} \fC_a(X^2) (\fC^*)^a(P^2),
\end{align}
 where $X = x^a \sfw_a$ and $P = u_a \sfw^a$.  Via the bracket \eqref{E:Lagrange}, $\{ x^i,u_i, f \}$ generate all of $\fg$.  (See Table \ref{F:explicit-sym}.)
 \end{theorem}
 
 \begin{proof}
 The $\fg_0$-invariant pairing $\fg_2 \times \fg_{-2} \to \fg_0$ surjects onto $\fz(\fg_0)$, so there exists $f \in \fg_2$ such that $\sfZ = [1,f] = f_u$.  Since $\sfZ = 2u - x^i u_i$, then $f = u (u - x^i u_i) + g(x^i,u_i)$.  By \eqref{E:Xf},
 \[
 \bS_f = (x^i u - g_{u_i}) \frac{d}{dx^i} + (u(u - x^i u_i) + g)\partial_u + \left( u_i (u - x^k u_k) + g_{x^i} \right)\partial_{u_i}.
 \]
 We require that $[\sfZ,f] = 2f$, which implies $x^i g_{x^i} + u_i g_{u_i} = 4g$, i.e.\ $g$ is homogeneous of degree 4.
 
 Fix $t \in W$ and let $\bV := \bV(1,t) \in \Gamma(\cC)$.  Let $[\bV,\bS_f] \in \Gamma(\cC)$ have components $(\rho^0,\rho^a, \mu_0, \mu_a)$ in the CS-framing $\{ \bX_i = \partial_{x^i} + u_i \partial_u, \bU^i = \partial_{u_i} \}$. Since $\bV(u - x^k u_k) = x^0 \frac{\fC(t^3)}{2} + \frac{3}{2} \fC(t,t,x)$, we have
 \begin{align*}
 \rho^0 = dx^0([\bV,\bS_f]) &= u + x^0 (u_0 - t^a u_a) - \bV(g_{u_0})\\
 \rho^a = dx^a([\bV,\bS_f]) &= -t^a u + x^a (u_0 - t^b u_b) - \bV(g_{u_a})\\
 \mu_0 = du_0([\bV,\bS_f]) &= \frac{\fC(t^3)}{2} (-u + x^k u_k + x^0 u_0) + \frac{3}{2} u_0 \fC(t,t,x) + \bV(g_{x^0})\\
 \mu_a = du_a([\bV,\bS_f]) &= -\frac{3\fC_a(t^2)}{2}(u - x^k u_k) + u_a\left(x^0 \frac{\fC(t^3)}{2} + \frac{3}{2} \fC(t,t,x)\right) + \bV(g_{x^a})
 \end{align*}
 Using \eqref{E:G-involutive} for $\widehat{T}_{[\bV]} \cV$, we row reduce $ \begin{psmallmatrix}  1 & 0 & \fC(t^3) & \frac{3}{2} \fC_b(t^2)\\
 0 & \delta_a{}^b & \frac{3}{2} \fC_a(t^2) & 3\fC_{ab}(t)\\
 \rho^0 & \rho^b & \mu_0 & \mu_b
 \end{psmallmatrix}
$ so that $[\bV,\bS_f] \in \widehat{T}_{[\bV]} \cV$ if and only if
 \begin{align}
  0 &= \mu_0 - \fC(t^3) \rho^0 - \frac{3}{2} \fC_a(t^2) \rho^a, \label{E:mu0}\\
  0 &= \mu_a - \frac{3}{2} \fC_a(t^2) \rho^0 - 3\fC_{ab}(t) \rho^b. \label{E:mua}
 \end{align}
  These equations are polynomial in $t$.  Extracting the $t$-degree 6,5,1,0 parts from \eqref{E:mu0} yields
  \begin{align*}
 g_{u_0 u_0} = g_{u_a u_0} = g_{x^0 x^a} = g_{x^0 x^0} = 0 \qRa g = A(x^a,u_b) + B(x^a) u_0 + \widetilde{B}(u_a) x^0 + \gamma x^0 u_0.
 \end{align*}
 Since $g$ is homogeneous of degree 4, then $A,B,\widetilde{B}$ are homogeneous of degrees $4, 3, 3$. In \eqref{E:mu0}[$t$-degree 4], set $x^0 = 0$, then differentiate with respect to $u_c$ to obtain $0 = \fC_a(t^2) \fC_b(t^2) A_{u_a u_b u_c}$.  Since $w_a = \fC_a(t^2)$ is arbitrary, then $A_{u_a u_b u_c} = 0$. Now consider \eqref{E:mua}[$t$-degree 1]:
 \begin{align*}
 0 &= 3\fC_{ab}(t) (g_{u_b x^0} - x^b u_0) - t^b g_{x^b x^a} = 3\fC_{ab}(t) (\widetilde{B}_{u_b} - x^b u_0) - t^b (A_{x^b x^a} + B_{x^b x^a} u_0),
 \end{align*}
 which splits according to $u_0$-degree.  Differentiation yields
 $B_{x^a x^b} = -3\fC_{ab}(x)$ and $A_{x^a x^b} = 3\fC_{abc} \widetilde{B}_{u_c}$.
 Since $A_{u_a u_b u_c} = 0$ and since $g$ is homogeneous of degree 4, this implies that:
 \begin{align*}
 B = -\frac{1}{2}\fC(X^3), \quad A = \frac{3}{2} \fC_a(X^2) \widetilde{B}_{u_a}.
 \end{align*}
 Observe that \eqref{E:mu0}[$t$-degree 4,3] is a contraction of \eqref{E:mua}[$t$-degree 3,2], so only the latter equations remain. Setting $x^b = 0$ in \eqref{E:mua}[$t$-degree 2] yields $\gamma = 0$. Now \eqref{E:mua}[$t$-degree 3] is at most linear in $x^0$, with
 \begin{align} \label{E:Bt}
 0 &= \left( \frac{\fC(t^3)}{2} \delta_a{}^b + \frac{3}{2} \fC_a(t^2) t^b\right) u_b - \frac{9}{2} \fC_c(t^2) \fC_{ab}(t) \widetilde{B}_{u_b u_c}
 \end{align}
 its $x^0$-degree 1 part.
 This determines the $x^0$-degree 0 part (differentiate by $u_d$ and contract with $\frac{1}{2} \fC_d(X^2)$) as well as the $t$-degree 2 part of \eqref{E:mua} (differentiate by $t^d$ and contract with $x^a$).  Thus, only \eqref{E:Bt} remains.
 
 Since $\widetilde{B}$ is cubic, write $\widetilde{B} = \beta^{abc} u_a u_b u_c$.  Contract \eqref{E:Bt} with $t^a$ and differentiate by $u_e$ to obtain:
 \begin{align}
 0 &= 2 \fC(t^3) t^e - 27\fC_b(t^2) \fC_c(t^2) \beta^{bce}, \qbox{i.e.} \beta(\fC(t^2)^2) = \frac{2}{27} \fC(t^3) t.
 \end{align}
 By \eqref{E:dual-cubic-id}, $\beta = \frac{1}{2} \fC^*$, so we obtain \eqref{E:top-slot}.  Finally, for $\bV := \bV(0,t)$ (which is a multiple of $\bU^0$), the condition $[\bV,\bS_f] \in \widehat{T}_{[\bV]} \cV = \tspan\{ \bU^i \}$ readily follows from observing that $f_{u_0 u_i} = 0$.
 \end{proof}
 
  \begin{table}[h]
 \[
 \begin{array}{|c|c|c|c|} \hline
 \fg_2 & & f & u(u - x^i u_i) - \frac{1}{2} \fC(X^3) u_0 + \frac{1}{2} \fC^*(P^3) x^0 + \frac{9}{4} \fC_a(X^2) (\fC^*)^a(P^2)\\ \hline
  \fg_1 & & [x^0,f] & x^0(u - x^i u_i) - \frac{1}{2} \fC(X^3)\\ 
 & & [x^a,f] & x^a(u - x^i u_i) + \frac{3}{2} (\fC^*)^a(P^2) x^0 + \frac{9}{2} \fC_b(X^2) (\fC^*)^{ab}(P)\\
 & & [u_0,f] & u u_0 - \frac{1}{2} \fC^*(P^3)\\
 & & [u_a,f] & uu_a + \frac{3}{2} \fC_a(X^2) u_0 - \frac{9}{2} \fC_{ab}(X) (\fC^*)^b(P^2)\\ \hline
 \fz(\fg_0) &  & [1,f] & \sfZ := 2u - x^i u_i \\ \hline
 \fg_0^{\Ss} & \ff_{1} & [x^a,[u_0,f]] & x^a u_0 - \frac{3}{2} (\fC^*)^a(P^2)\\ 
 & \ff_0 & [x^0,[u_0,f]] - \frac{1}{2} \sfZ & \sfZ^{(0)} := \frac{3}{2} x^0 u_0 + \frac{1}{2} x^c u_c \in \fz(\ff_0)\\
 & \ff_0 & [x^a,[u_b,f]] - \delta^a{}_b (\frac{1}{2}\sfZ + \frac{1}{3} \sfZ^{(0)}) & \psi^a{}_b := x^a u_b + \frac{1}{3} \delta^a{}_b x^c u_c - 9 \fC_{bc}(X) (\fC^*)^{ac}(P) \\
 & \ff_{-1} & [u_a,[x^0,f]] & u_a x^0 + \frac{3}{2} \fC_a(X^2)\\ \hline
 \fg_{-1} & & & x^i, u_i \\ \hline
 \fg_{-2} & && 1 \\ \hline
 \end{array}
 \]
 \caption{Generating functions for any complex simple $\fg$ not of type $A$ or $C$}
 \label{F:explicit-sym}
 \end{table}
 
 Assign weighted degree $+1$ to $x^i, u_i$ and $+2$ to $u$, so $\fg_k$ are polynomials of weighted degree $k+2$. 
  
 \begin{cor} \label{C:g-embed} Let $\fg$ be a complex simple Lie algebra not of type $A$ or $C$.  Then $\fg$ embeds into the space of polynomials in $x^i, u, u_i$ $(0 \leq i \leq n-1)$ of weighted degree $\leq 4$, equipped with the Lagrange bracket.
 \end{cor}

 \begin{remark}
 Theorem \ref{T:sym-flat} and Corollary \ref{C:g-embed} are also valid for $G = A_{n+1}$ with $\fC = 0$ and $1 \leq i \leq n$, i.e.\ the $0$-th coordinate is not distinguished (\S \ref{S:typeA}).  It is also valid for $G = A_1$: $\{ 1, 2u, u^2 \}$ is a standard $\fsl_2$-triple.
 \end{remark}

 All generating functions for $\fg$ are given as linear combinations of those specified in Table \ref{F:explicit-sym}.
 \begin{itemize}
 \item $\sfZ = \sfZ_j$ and $\sfZ^{(0)} = \sum_{i \in N(j)} \sfZ_i$, where $N(j)$ are neighbouring nodes to the contact node $j$.
 \item $\psi^a{}_b \in \ff_0$ are the only functions quadratic in $\{ x^c, u_c \}$ and independent of $u, x^0, u_0$.  Hence, their span is closed under the Poisson bracket $[f,g] = f_{x^c} g_{u_c} - g_{x^c} f_{u_c}$, i.e.\ restriction of the Lagrange bracket.
 \end{itemize}
 
 \begin{cor} \quad
 \begin{enumerate}
 \item[(i)] $G = F_4, E_6, E_7, E_8$: $\{ \psi^a{}_b \}_{1 \leq a,b \leq \dim(W)}$ spans $\ff_0^{\Ss} \cong A_2, A_2 \times A_2, A_5, E_6$ respectively; $\dim(\fz(\ff_0)) = 1$.
 
 \item[(ii)] $G = \SO_{m+6}$, $m \geq 1$: $W = \JS_m = \bbC^m\op \bbC$, $\psi^\infty{}_\sfa = \psi^\sfa{}_\infty = 0$, $\sfZ^{(\infty)} := \psi^\infty{}_\infty = \frac{4}{3} x^\infty u_\infty - \frac{2}{3} x^\sfa u_\sfa$.
 \begin{itemize}
 \item $m \geq 3$: $\ff_0^{\Ss} \cong \fso_m$ is the span of all $x^\sfa u_\sfb - x_\sfb u^\sfa$.  (See the adapted basis $\{ \sfw_\sfa \}_{\sfa = 1}^m$ in \S \ref{S:LM} and use $\langle \cdot, \cdot \rangle$ and its inverse to lower and raise indices.)  Also, $\fz(\ff_0) = \tspan \{ \sfZ^{(0)},\sfZ^{(\infty)} \}$.
 \item $m=1$ or $m=2$: $\ff_0^{\Ss} = 0$ and $\fz(\ff_0) = \tspan \{ \sfZ^{(0)},\sfZ^{(\infty)}, \sfZ^{(1)} \}$, where $\sfZ^{(1)} := x^1 u_1 - x_1 u^1$.
 \end{itemize}
  
 \end{enumerate}
 \end{cor}
 
 \begin{proof} Table \ref{F:extended-magic} yields the assertions for $\ff_0^{\Ss}$.  For (ii), if $q_{\sfa\sfb}$ are the components of $\langle \cdot, \cdot \rangle$ with respect to $\{ \sfw_a \}$, then for $t = (v,\lambda)$, we have $\fC(t^3) = \fC_{abc} t^a t^b t^c = q_{\sfa\sfb} v^\sfa v^\sfb \lambda$, so $\fC_{\sfa\sfb}(t) = \frac{1}{3} q_{\sfa\sfb} \lambda$ and $\fC_{\sfa\infty}(t) = \frac{1}{3} q_{\sfa\sfb} v^\sfb$.  Consulting Table \ref{F:dual-cubic}, we have for $s^* = (v^*,\mu)$, $(\fC^*)^{\sfa\sfb}(s^*) = \frac{1}{3} q^{\sfa\sfb} \mu$ and $(\fC^*)^{\sfa\infty}(s^*) = \frac{1}{3} q^{\sfa\sfb} v_\sfb$.  Substitution into $\psi^a{}_b$ then yields the result.  (In particular, $\psi^\sfa{}_\sfb + \frac{1}{2} \delta^\sfa{}_\sfb \sfZ^{(\infty)} = x^\sfa u_\sfb - x_\sfb u^\sfa$.)
 \end{proof}

 \begin{example}[$\fg = F_4$, $\ff_0^{\Ss} = A_2$]  On $W = \cJ_3(\bbR_\bbC)$, $\fC(t^3) = \det(t)$.  Let $\sfe^i{}_j$ denote the $3 \times 3$ matrix with 1 in the $(i,j)$-position and 0 otherwise.  Consider the ordered basis of $W = \cJ_3(\bbR_\bbC)$ given by
 \[
 \sfw_1 = \sfe^1{}_1, \quad 
 \sfw_2 = \sfe^2{}_2, \quad
 \sfw_3 = \sfe^3{}_3, \quad
 \sfw_4 = \sfe^1{}_2 + \sfe^2{}_1, \quad
 \sfw_5 = \sfe^1{}_3 + \sfe^3{}_1, \quad
 \sfw_6 = \sfe^2{}_3 + \sfe^3{}_2.
 \]
 This basis determines coordinates $\{ x^a \}$ on $W$ and dual coordinates $\{ u_a \}$ on $W^*$.
 A Lie algebra isomorphism from $A_2 = \fsl_3$ to $\tspan\{ \psi^a{}_b \}_{1 \leq a,b \leq \dim(W) }$ is given by
 \begin{align*}
 \begin{array}{l}
 \begin{array}{l@{\,\,\mapsto\,\,}c}
 \diag(\frac{2}{3},-\frac{1}{3},-\frac{1}{3}) & 
 \psi^1{}_1 = \frac{4}{3} x^1 u_1 - \frac{2}{3} x^2 u_2 - \frac{2}{3} x^3 u_3 + \frac{1}{3} x^4 u_4 + \frac{1}{3} x^5 u_5  - \frac{2}{3} x^6 u_6; \\
 \diag(\frac{1}{3},\frac{1}{3},-\frac{2}{3}) &
 -\psi^2{}_2 = \frac{2}{3} x^1 u_1 - \frac{4}{3} x^2 u_2 + \frac{2}{3} x^3 u_3 - \frac{1}{3} x^4 u_4 + \frac{2}{3} x^5 u_5  - \frac{1}{3} x^6 u_6;
 \end{array}\\
 \begin{array}{c@{\mapsto}c@{\qquad\qquad}c@{\mapsto}c}
 \sfe^1{}_2 & \psi^3{}_5 = 2x^5 u_1+ x^6 u_4 + x^3 u_5; & 
 \sfe^2{}_1 & \psi^1{}_5 = x^1 u_5 + x^4 u_6 + 2 x^5 u_3;\\
 \sfe^2{}_3 & \psi^2{}_6 = x^4 u_5 + x^2 u_6 + 2 x^6 u_3; & 
 \sfe^3{}_2 & \psi^3{}_6 = x^5 u_4 + 2 x^6 u_2 + x^3 u_6;\\
 \sfe^1{}_3 & \psi^2{}_4 = 2x^4 u_1 + x^2 u_4 + x^6 u_5; & 
 \sfe^3{}_1 & \psi^1{}_4 = x^1 u_4 + 2 x^4 u_2 + x^5 u_6.
 \end{array}
 \end{array}
 \end{align*}
 \end{example}

 The following result will be used in \S \ref{S:involutive} and \S \ref{S:Monge}.
 
 \begin{prop} \label{P:skew-cubic} Consider the map $\Psi: \tEnd(W) \to W^* \otimes \bigwedge^2 W^*$ given by $A^a{}_b \mapsto \fC_{ab[c} A^a{}_{d]}$. Then
 \[
 \ker(\Psi) = \left\{ \begin{array}{ll}
 \tspan\{ \id_W \}, & \mbox{if } W \neq \JS_1;\\
 \tspan\{ \id_W, (\sfw_1)^* \otimes \sfw_\infty \}, & \mbox{if } W = \JS_1.\\
 \end{array} \right.
 \]
 \end{prop}
 
 \begin{proof}  Let $A \in \ker(\Psi)$.  The $W = \GJ$ case is trivial and the $W = \JS_1$ case follows easily.
 \begin{enumerate}
  \item $W = \JS_m$, $m \geq 2$: If $b=d=\infty \neq \sfc$, then $A^{\sfc'}{}_\infty = 0$, where $\sfc' = m+1-\sfc$. Let $1\leq \sfb,\sfc \leq m$.
 \begin{itemize}
 \item $\sfc \neq \sfd = \sfb'$: We have $A^\infty{}_\sfc = 0$. (Note $m \geq 2$ was used here.)
 \item $d=\infty$: If $\sfc \neq \sfb'$, then $A^{\sfb'}{}_\sfc = 0$.  If $\sfc = \sfb'$, then $A^\infty{}_\infty = A^\sfc{}_\sfc$ (no sum).
 \end{itemize}

 \item $W = \cJ_3(\bbA)$, $\bbA \neq \underline{0}$: $\tEnd(W)$ contains $\ff_0$ and the Cartan product $\tEnd_0(W)$ of $W$ and $W^*$.
  \[
 \begin{array}{|c|cccc|}\hline
 \ff_0^{\Ss} & A_2 & A_2 \times A_2 & A_5 & E_6\\
 \ff_0^{\Ss}\mbox{-weight of $W$} & 2\lambda_1 & \lambda_1 + \lambda_1' & \lambda_2 & \lambda_6\\
 \ff_0^{\Ss}\mbox{-weight of $\tEnd_0(W)$} & 2\lambda_1 + 2\lambda_2 & \lambda_1 + \lambda_2 + \lambda_1' + \lambda_2' & \lambda_2 + \lambda_4 & \lambda_1 + \lambda_6 \\
 \dim(\tEnd_0(W)) & 27 & 64 & 189 & 650\\ \hline
 \end{array}
 \]
 By comparing dimensions, it follows that $\tEnd(W)  \cong \bbC \op \ff_0^{\Ss} \op \tEnd_0(W)$ as $\ff_0^{\Ss}$-irreps.  By Schur's lemma, it suffices to verify that $\Psi \neq 0$ on $\ff_0^{\Ss}$ and $\tEnd_0(W)$.
 \begin{itemize}
 \item $\tEnd_0(W)$:  take highest weight vectors for $W$ and $W^*$, so their product is a highest weight vector $A \in \tEnd_0(W)$.  In an adapted weight basis, we have $A^a{}_b = 0$ except for $A^1{}_{n-1} \neq 0$.  Injectivity of $\fC : W \to S^2 W^*$ implies that $\Psi(A) \neq 0$. 
 \item $\ff_0^{\Ss}$: From \S \ref{S:LM}, we know $[\ff_{-1}, V_0] = V_{-1} = W$.  Explicitly, $V_0 = \tspan\{ u_0 \}$, $W = V_{-1} = \tspan\{ u_a \}$, $\ff_{-1} = \tspan\{ u_a x^0 + \frac{3}{2} \fC_a(X^2) \}$, and $[u_a x^0 + \frac{3}{2} \fC_a(X^2),u_0] = u_a$.  Referring to Table \ref{F:explicit-sym}, let $A = \psi^a{}_b \in \ff_0$, so $[A,u_d] = A^k{}_d u_k$, where $A^k{}_d = \delta^a{}_d \delta^k{}_b + \frac{1}{3} \delta^a{}_b \delta^k{}_d - 9 \fC_{bed} (\fC^*)^{aek}$.  Evaluate $\Psi(A)$.  Use \eqref{E:dual-cubic-id} and set $a =d \neq c$.  Then
 \begin{align*}
 \fC_{kh[c} A^k{}_{d]} &= \fC_{kh[c} \delta^a{}_{d]} \delta^k{}_b + \frac{1}{3} \fC_{kh[c} \delta^k{}_{d]} \delta^a{}_b - 9 \fC_{kh[c} \fC_{d]be}(\fC^*)^{aek} = \fC_{bh[c} \delta_{d]}{}^a = \frac{1}{2} \fC_{bhc} \not\equiv 0.
 \end{align*}
 \end{itemize}
 \end{enumerate}
 \end{proof}

 \subsection{The PDE system $\cE$}
 \label{S:E-PDE}

Let $G \neq A_\ell, C_\ell$.  Consider the PDE $\cE \subset J^2$ for the flat $G$-contact structure.  As in \S \ref{S:contact}, we obtain a PD-manifold $(\cE; \cD, \widetilde\cC)$ whose symmetries correspond to external symmetries of $\cE \subset J^2$.
  
 We have local coordinates $(x^i,u,u_i,t^a)$ on $\cE$ adapted to $\cE \to M^{2n+1}$, so $\dim(\cE) = 3n$.  While $\widetilde\cC$ is given by $\fann\{ du - u_i dx^i \} = \tspan\{ \partial_{x^i} + u_i \partial_u, \partial_{u_i}, \partial_{t^a} \}$, $\cD$ has rank $2n-1$ and is spanned by
 \begin{align} \label{E:E-D-framing}
 \bT_a = \partial_{t^a}, \quad
 \bX_i = \partial_{x^i} + u_i \partial_u + u_{ij} \partial_{u_j}, \qbox{where} 
 (u_{ij})
 = \left( \begin{array}{c|c} \fC(t^3) & \frac{3}{2} \fC_b(t^2)\\[0.05in] \hline \frac{3}{2} \fC_a(t^2) & 3\fC_{ab}(t)  \end{array} \right).
 \end{align}
 Let us consider the subordinate structure $(\cE,\cD)$.

 \subsubsection{Involutivity}
 \label{S:involutive}

 The dual Pfaffian system to $\cD$ is $\cI = \tspan\{ \sigma, \theta_i \}$, where
 \begin{align} \label{E:EV-EDS}
 \sigma = du - u_i dx^i, \quad
 \theta_0 = du_0 - \fC(t^3) dx^0 - \frac{3}{2} \fC_a(t^2) dx^a, \quad
 \theta_a = du_a - \frac{3}{2} \fC_a(t^2) dx^0 - 3 \fC_{ab}(t) dx^b.
 \end{align}
 Defining $\widetilde\theta_0 = \theta_0 - t^a \theta_a$, $\omega^0 = dx^0$, 
$\omega^b = 3(dx^b + t^b dx^0)$, and $\pi^c = dt^c$, we have
 \begin{align} \label{E:EV-streq}
 d\sigma \equiv 0, \quad d\widetilde\theta_0 \equiv 0, \quad 
 d\theta_a \equiv \fC_{abc} \omega^b \wedge \pi^c \qquad \mod \cI.
 \end{align}
 Letting $\cJ = \tspan\{ \sigma, \theta_i, \omega^i \}$, then $(\cI,\cJ)$ is a linear Pfaffian system, i.e. $d\cI \equiv 0 \mod \cJ$.  Its corresponding tableau is involutive (in the sense of Cartan--K\"ahler \cite{Olv1995, IL2003}) only in two cases:
 
 \begin{theorem} \label{T:involutive} Given $G \neq A_\ell,C_\ell$ a complex simple Lie group, the tableau for $(\cI,\cJ)$ is involutive if and only if $G = G_2$ or $B_3$.  These cases have single nonzero Cartan character $s_1 = 1$ or $s_1 = 2$ respectively.
 \end{theorem}
 
 \begin{proof} The {\em degree of indeterminacy} $r^{(1)}$ of $(\cI,\cJ)$ is the dimension of the space of all $A^c{}_i$ such that the replacements $\pi^c \mapsto \pi^c + A^c{}_i \omega^i$ preserve \eqref{E:EV-streq}, i.e.\ $0 = \fC_{ac[b} A^c{}_{i]}$.   (Clearly, $A^c{}_i = \delta^c{}_i \lambda$ solves this.)  Thus, $A^c{}_0 = 0$, so $0 = \fC_{ac[b} A^c{}_{d]}$ remains.  By Proposition \ref{P:skew-cubic}, $r^{(1)} = 1$ if $G \neq B_3$ and $r^{(1)} = 2$ if $G = B_3$.

The only non-trivial part of the tableau associated to $(\cI,\cJ)$ is the 1-form valued submatrix with symmetric entries $\cT_{ab} = \fC_{abc} \pi^c$.  The Cartan characters $s_1, s_2,...$ are obtained as follows: $s_1$ is the maximal number of linearly independent 1-forms in the first column, $s_2$ is the maximal number of independent  1-forms in the second column that are independent of the first column, etc.  These characters are computed under the assumption of working in a generically chosen basis for the tableau, and we have $s_1 \geq s_2 \geq ... \geq 0$.  Cartan's test for involutivity of $(\cI,\cJ)$ is that $s_1 + 2 s_2 + 3 s_3 + ... = r^{(1)}$.

 Suppose $G \neq B_3$, so $r^{(1)} = 1$.  Involutivity forces $s_1 = 1$ and $s_2 = s_3 = ... = 0$, so all entries of the tableau are linearly dependent.  But $\fC : W \to S^2 W^*$ is injective (Lemma \ref{L:C-inj}), so $\dim(W) = 1$, i.e.\ $G = G_2$.
 
 Suppose $G = B_3$, so $r^{(1)} = 2$.  Thus, $W = \JS_1$ and the only non-trivial component of $\fC$ is $\fC_{11\infty} = 1$, so $\cT = \begin{psmallmatrix} \pi^\infty & \pi^1\\ \pi^1 & 0\end{psmallmatrix}$.  The only nonzero Cartan character is $s_1 = 2$ and Cartan's test is satisfied.
 \end{proof}

 In the involutive cases $W = \GJ$ and $W = \JS_1$, the general integral manifold of \eqref{E:EV-EDS} will depend on {\em $s_1$ functions of one variable}. In all non-involutive cases, we need to prolong the system: consider the bundle $\widetilde\cE \to \cE$ with 1-dimensional fibres, fibre coordinate $\lambda$, equipped with $\widetilde\cI$ consisting of (the pullback of) $\cI$ together with the forms $\widetilde\pi^c := \pi^c + \lambda \omega^c$.  Then $d\cI \equiv 0 \mod \widetilde\cI$, but
 \[
 d\widetilde\pi^c \equiv (d\lambda + 3\lambda^2 dx^0) \wedge \omega^c \quad\mod \widetilde\cI.
 \]
 This system has torsion, so we must restrict to any submanifold $\widetilde\cE_0 \subset \widetilde\cE$ on which $d\lambda + 3\lambda^2 dx^0 = 0$.  Then $\widetilde\cI|_{\widetilde\cE_0}$ is a rank $2n$ Frobenius system on the $3n$-manifold $\widetilde\cE_0$.  Together with an additional arbitrary constant parametrizing the possible $\widetilde\cE_0$, the general integral manifold of $\cI$ will depend on $2n+1$  arbitrary {\em constants}.
 
  \subsubsection{Homogeneous space description}

 Let $G \neq A_\ell, C_\ell$ and $G^{\ad} \cong G/P_j$.  In the Dynkin diagram of $G$, let $N(j)$ denote the neighbouring nodes to the contact node $j$, and let $\sfZ_{N(j)} = \sum_{i \in N(j)} \sfZ_i$.  Refine the $\sfZ_j$-grading on $\fg$ with the $(\sfZ_{N(j)},\sfZ_j)$-bigrading.  From \eqref{E:f-decomp}--\eqref{E:V-decomp}, we have:
 \begin{align} \label{E:bigrading}
 \ff_{-1} = \fg_{-1,0}, \quad 
 \fg_{\geq 0,0} = \fz(\fg_0) \op \fq, \quad 
 \ff_1 = \fg_{1,0}, \quad 
 V_{-k} = \fg_{-k,-1} \quad (k=0,1,2,3).
 \end{align}
 The symplectic form on $V$ pairs $V_0$ with $V_{-3}$ and $V_{-1}$ with $V_{-2}$.  This yields a 1-dimensional subspace $\fg_{-3,-2}$ transverse to $V$.  (See Figure \ref{F:G2-rootdiag} for the $G_2$-case.) 
 
 \begin{figure}[h]
 \framebox{\includegraphics[width=15cm]{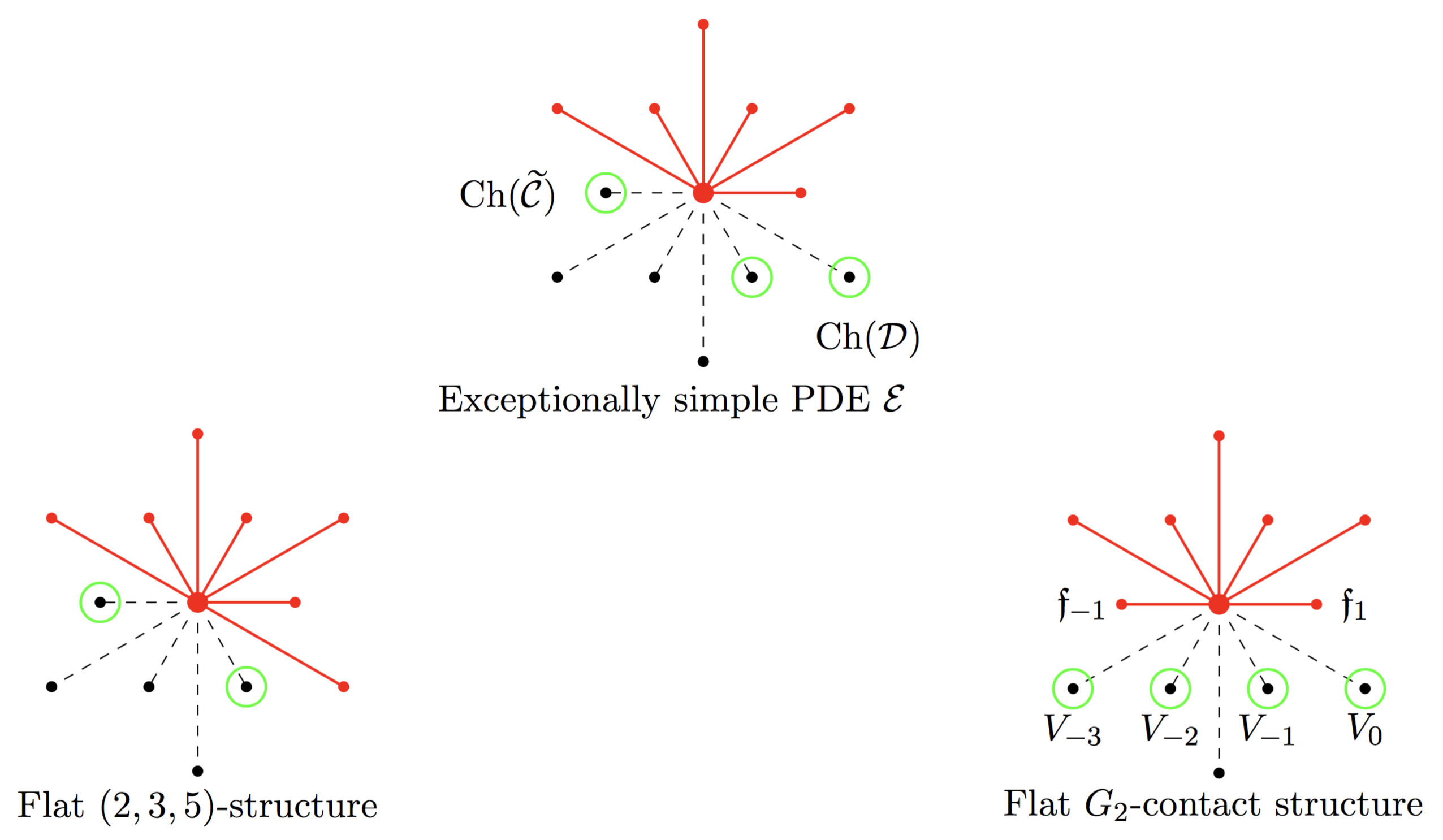}}
 \caption{Some geometric structures encoded on the $G_2$-root diagram}
 \label{F:G2-rootdiag}
 \end{figure}
 
 \begin{prop} \label{P:bigrading}
 Let $G \neq A_\ell, C_\ell$ and consider the PD-manifold $(\cE;\cD,\widetilde\cC)$ for the flat $G$-contact structure on $M = G^{\ad} \cong G/P_j$.  As $G$-homogeneous spaces, $\cE \cong G/P_{B(j)}$, where $B(j) = \{ j \} \cup N(j)$.  Fix $o \in G/P_j$, 
 $l \in \bbP(\cC_o)$ the highest weight line, and $p = \widehat{T}_l \cV_o \in \tLG(\cC_o)$.
 With respect to the $(\sfZ_{N(j)},\sfZ_j)$-bigrading,
 \begin{itemize}
 \item $\cD_p \cong \fg_{-1,0} \op \fg_{0,-1} \op \fg_{-1,-1} \cong W \op \fg_{0,-1} \op W$.  In particular, $\fg_{0,-1} \subset \Ch(\cD)_p$.
 \item $\widetilde\cC_p \cong \fg_{-1,0} \op \fg_{0,-1} \op \fg_{-1,-1} \op \fg_{-2,-1} \op \fg_{-3,-1} \cong W \op \bbC \op W \op W^* \op \bbC$.
 \end{itemize}
 With respect to the $\sfZ_{B(j)}$-grading, these are $\fg_{-1} \op \fg_{-2}$ and $\fg_{-1} \op ... \op \fg_{-4}$ respectively.
 \end{prop}
 
 \begin{proof}
 The parabolic $Q \subset F = G_0^{\Ss}$ is the stabilizer of $l$.  Since the map $\cV_o \to \tLG(\cC_o)$ is an embedding, then $Q \subset F$ is also the stabilizer of $p = \widehat{T}_l \cV_o$.  As remarked in \S \ref{S:SA-var}, the crossed nodes for $Q$ are $N(j)$.  Hence, the stabilizer in $G$ of $p$ is $P_{B(j)}$, so $\cE \cong G/P_{B(j)}$.
 
 In the Lagrange--Grassmann bundle $M^{(1)}$ over $(M,\cC)$, we restrict $\cC^{(1)}$ and $(\cC^{(1)})^{-2}$ to $\cE$ to obtain $\cD$ and $\widetilde\cC$ respectively.  Referring to \eqref{E:bigrading}, we have $p = V_0 \op V_{-1}$ and $\cC_o$ is identified with $V = V_0 \op V_{-1} \op V_{-2} \op V_{-3}$.  The vertical subspace at $p \in M^{(1)}$ is isomorphic to $\ff_{-1} = \fg_{-1,0}$.  By definition, the pullback under the projection $\cE \to M$ of: (i) $\cC_o$ yields $\widetilde\cC_p$, and (ii) $p \subset \cC_o$ yields $\cD_p$.  This yields the stated decompositions.
 \end{proof}

 We emphasize that the subordinate structure $(\cE,\cD)$:
 \begin{itemize}
 \item is {\em not} the underlying structure for a $(G,P_{B(j)})$-geometry.  (It is not finite-type since $\Ch(\cD) \neq 0$.)
 \item has no distinguished vertical subspace (corresponding to $\ff_{-1} = \fg_{-1,0}$).  This is determined from the additional data of $\widetilde\cC$.  Namely, this vertical subspace is $\Ch(\widetilde\cC) = \Ch(\cC^{(1)}) \cap T\cE$.
 \end{itemize}
 
 \newpage
 \subsection{Cauchy characteristics and second-order Monge geometries}
 \label{S:Cauchy}

 \subsubsection{Cauchy characteristic reduction}
 
 We know that $(\cE,\cD)$ admit Cauchy characteristics.  More precisely:
 \begin{prop}
 $\rnk(\Ch(\cD)) = 1$ (so $\Ch(\cD)_p = \fg_{0,-1}$) and $\Ch(\cD) \not\subset \Ch(\widetilde\cC)$ is spanned by
 \begin{align} \label{E:Cauchy-char}
  \bZ = \partial_{x^0} + u_0 \partial_u - t^a \left( \partial_{x^a} + u_a \partial_u\right) - \frac{\fC(t^3)}{2} \partial_{u_0} - \frac{3\fC_a(t^2)}{2} \partial_{u_a}.
 \end{align}
 \end{prop}
 
 \begin{proof}  From \eqref{E:E-D-framing}, the non-trivial commutator relations are
 \begin{align*}
 [\bT_a,\bX_b] = 3\fC_{ab}(t) \partial_{u_0} + 3\fC_{abc} \partial_{u_c}, \quad
 [\bT_a,\bX_0] = t^b [\bT_a,\bX_b].
 \end{align*}
 By injectivity of $\fC : W \to S^2 W^*$, $\rnk(\cD^{-2}) = 3n-2$.  
 By \eqref{E:EV-streq}, $\cE^{(1)} \to \cE$ is onto, so by \eqref{E:PD-Cauchy}, $\rnk(\Ch(\cD)) = 1$.  Since $[\bZ, \bT_a] = \bX_a$ and $[\bZ,\bX_0] = [\bZ, \bX_a] = 0$, then $\bZ \in \Ch(\cD)$.  Also, $\bZ \not\in \Ch(\widetilde\cC)$.
 \end{proof}

 Tautologically, $\bbP(\Ch(\cD))$ {\em is} $\cV \subset \bbP(\cC)$ on $M$. The vector field $\bZ$ has the $3n-1$ invariants
  \begin{align*}
 X^a &= x^a + t^a x^0, \quad 
 U = u - (u_0 - t^a u_a) x^0 + \frac{\fC(t^3)}{2} (x^0)^2, \quad 
 U_a = u_a + \frac{3\fC_a(t^2)}{2} x^0,\\
 T^a &= t^a, \quad  Z = u_0 + \frac{\fC(t^3)}{2} x^0,
 \end{align*}
 which yields local coordinates $(X^a, U, U_a, T^a, Z)$ on the leaf space $\overline\cE$.  Pulling back $\cI$ by the section $\sigma : \overline\cE \to \cE$ determined by $x^0 = 0$ yields the differential system 
 \begin{align} \label{E:CCred}
 \omega = dZ - \frac{3}{2} \fC_a(T^2) dX^a, \quad \theta = dU - U_a dX^a, \quad \theta_a = dU_a - 3\fC_{ab}(T) dX^b, \quad 1 \leq a,b \leq n-1.
 \end{align}
 Its dual vector distribution $\overline\cD$ is spanned by
 \begin{align} \label{E:D-red}
  \partial_{T^a}, \quad \bY_a = \partial_{X^a} + U_a \partial_U + 3\fC_{ab}(T) \partial_{U_b} + \frac{3}{2} \fC_a(T^2) \partial_Z.
 \end{align}
 Note $[\partial_{T^a}, \bY_b] = 3\fC_{abc} \partial_{U_c} + 3\fC_{ab}(T) \partial_Z$.  Since $\fC : W \to S^2 W^*$ is injective, we may take $\fm_{-2} \cong \cD^{-2} / \cD$ to be spanned by $\partial_{U_c} + T^c \partial_Z \,\,(\!\mod \cD)$. We obtain the symbol algebra $\fg_- = \fg_{-1} \op \fg_{-2} \op \fg_{-3}$, where
 \begin{align} \label{E:PNj-symbol}
 \fg_{-1} \cong W \op W \cong W \otimes \bbC^2, \qquad \fg_{-2} \cong W^*, \qquad \fg_{-3} \cong \bbC^2.
 \end{align}
 With respect to the bracket $\bigwedge^2 \fg_{-1} \to \fg_{-2}$, each copy of $W$ in $\fg_{-1} \cong W \op W$ is isotropic, while cross terms yield a map $W \otimes W \to W^*$ expressed via $\fC$.  The bracket $\fg_{-1} \times \fg_{-2} \to \fg_{-3}$ is the natural contraction.

 More abstractly, referring to the $(\sfZ_{N(j)},\sfZ_j)$-bigrading used in Proposition \ref{P:bigrading}, we have $\Ch(\cD)_p = \fg_{0,-1}$, so quotienting by this only retains the $\sfZ_{N(j)}$-grading and we may re-express \eqref{E:PNj-symbol} as
 \begin{align}
 \fg_{-1} = \fg_{-1,0} \op \fg_{-1,-1}, \qquad \fg_{-2} = \fg_{-2,-1}, \qquad \fg_{-3} = \fg_{-3,-1} \op \fg_{-3,-2}.
 \end{align}
 Thus, the symbol algebra for the distribution $\overline\cD$ on $\overline\cE$  corresponding to $\fg_{-1}$ matches that for the underlying structure for (regular, normal) $(G,P_{N(j)})$-geometries.  (From \cite{Yam1993}, this structure comes from the filtration; no additional structure group reduction is required.)  We refer to these as {\em second-order Monge geometries}.

 \begin{prop} \label{E:CC-red}
 Let $G \neq A_\ell, C_\ell$.  The structure $(\overline\cE,\overline\cD)$ in \eqref{E:CCred} is the flat model for $(G,P_{N(j)})$-geometries.   If moreover $G \neq G_2, B_3$, then any (regular, normal) $(G,P_{N_{(j)}})$-geometry is (locally) flat.
 \end{prop}
 
 \begin{proof} All symmetries of $(\cE;\cD, \widetilde\cC)$ preserve $\Ch(\cD)$, so are projectable over $\overline\cE$.  Since $\Ch(\cD) \cap \Ch(\widetilde\cC) = 0$, then $\fg$ injects into the symmetry algebra of $(\overline\cE,\overline\cD)$.  But all $(G,P_{N(j)})$-geometries have symmetry dimension at most $\dim(\fg)$.  The flat model is uniquely maximally symmetric, which implies the first claim.
 
 The second claim is due to Yamaguchi \cite[pg.313]{Yam1999}.  Namely, almost all $(G,P_{N(j)})$ satisfy $H^2_+(\fg_-,\fg) = 0$, hence $\kappa_H = 0$ for regular, normal geometries.  The only exceptions are $(G_2,P_1)$ and $(B_3,P_{1,3})$.
 \end{proof}
 
 \subsubsection{Second order Monge equations and their solutions}
 \label{S:Monge}
 
 Integral manifolds of \eqref{E:CCred}, i.e.\ submanifolds upon which $\omega,\theta,\theta_a$ vanish, have maximal dimension $\dim(W)$ and we may take these to be parametrized by $X^a$.  These are the solutions to the {\em second order Monge equations} \eqref{E:Monge}, where $T^c = T^c(X^e)$.  On integral manifolds, $d\omega, d\theta, d\theta_a$ also vanish.  Since $d\omega = T^a d\theta_a$, the compatibility condition is
 \begin{align} \label{E:Monge-comp}
 \fC_{ac[b} T^c{}_{,e]} = 0, \qbox{where} T^c{}_{,e} = \frac{\partial T^c}{\partial X^e}.
 \end{align}

 \begin{prop}  If $W \neq \GJ$ and $W \neq \JS_1$, the only solutions to \eqref{E:Monge-comp} are $T^c = \lambda X^c + \mu^c$, where $\lambda$ and $\mu^c$ are constants.  In these cases, the solution of \eqref{E:Monge} is
 \begin{align}
 Z &= \frac{\lambda^2}{2} \fC(X^3) + \frac{\lambda}{2} \fC_a(X^2) \mu^a + \fC_{ab}(X) \mu^a \mu^b + const\\
 U &= \frac{\lambda}{2} \fC(X^3) + \frac{1}{2} \fC_a(X^2) \mu^a + \nu_a X^a + const,
 \end{align}
 depending on $2n+1$ arbitrary constants.
 \end{prop}
 
 \begin{proof}
 By Proposition \ref{P:skew-cubic}, $T^c{}_{,e} = \lambda \delta^c{}_e$, so $T^c = \lambda X^c + \mu^c$ as claimed, and the rest easily follows.
 \end{proof}
 
  For $G = G_2$, i.e.\ $W = \GJ$, \eqref{E:Monge-comp} is trivial.  Eliminating $T$, \eqref{E:Monge} becomes $Z' = \frac{1}{2} (U'')^2$ (equivalent to Hilbert--Cartan).  When $G = B_3$, i.e.\ $W = \JS_1$, $\fC(t^3) = v^2\lambda$ for $t = (v,\lambda)$, and \eqref{E:Monge} becomes 
 \begin{align} \label{E:B3-Monge-flat}
  Z_x = U_{xx} U_{xy}, \quad Z_y = \frac{1}{2} (U_{xy})^2, \quad U_{yy} = 0.
  \end{align}
 This has solution $Z = \frac{1}{2} f'(x)^2 y + \int f'(x) g''(x) dx$, $U = f(x) y + g(x)$.

 \subsection{The parabolic Goursat PDE $\cF$}
 \label{S:Goursat}

 Let $G \neq A_\ell, C_\ell$.  For the flat $G$-contact structure, consider $\cF \subset M^{(1)}$ and its associated PD-manifold $(\cF;\cD,\widetilde\cC)$.  We now verify that $\cF$ has parabolic Goursat type and show how the first-order covariant system $\cN$ \cite{Yam1999} leads to the sub-adjoint variety field $\cV$.
 
 The hypersurface $\cF \subset M^{(1)}$ has local coordinates $(x^i,u,u_i,u_{ab},t^a)$ adapted to $\cF \to M^{2n+1}$.  While $\widetilde\cC = \fann\{ du - u_i dx^i \} = \tspan\{ \partial_{x^i} + u_i \partial_u, \partial_{u_i}, \partial_{u_{ab}}, \partial_{t^a} \}$, $\cD$ has rank $2n-1$ and is spanned by
 \begin{align} \label{E:F-D-framing}
 \partial_{t^a}, \quad \partial_{u_{ab}}, \quad
 \partial_{x^i} + u_i \partial_u + u_{ij} \partial_{u_j}, \qbox{where} 
 \left\{ \begin{array}{r@{\,}l} 
 u_{00} &= t^a t^b u_{ab} - 2\fC(t^3)\\ 
 u_{a0} &= t^b u_{ab} - \frac{3}{2} \fC_a(t^2)
 \end{array}. \right.
 \end{align}
 Off of $u_{ab} = 3\fC_{ab}(t)$, i.e.\ $\cE \subset \cF$, we have $\cD^{-2} = \widetilde\cC$, so $\Ch(\cD) = 0$ by \eqref{E:PD-Cauchy}.  (We verify that $\cF^{(1)} \to \cF$ is onto.)  Hence, $(\cF,\cD)$ and $(\cF;\cD,\widetilde\cC)$ share the same symmetries (which are the external symmetries of $\cF \subset M^{(1)}$; see \S \ref{S:contact}).  The {\em symbol algebra}\footnote{The symbol algebra of a PD-manifold should not be confused with the symbol algebra of a distribution.} of $(\cF;\cD,\widetilde\cC)$ at $v \in \cF$ is the subalgebra of $\fg_-$ in \eqref{E:J2-symbol} given by 
 \begin{align} \label{E:F-symbol}
 \fs(v) = \fs_{-1}(v) \op \fs_{-2}(v) \op \fs_{-3}(v) = (L \op \fr(v)) \op L^* \op \bbC,
 \end{align}
 where $\fr(v) \subset S^2 L^*$ has codimension one.  This corresponds to the vertical subspace spanned by differentiating the parametric equations for $\cF$ by $\partial_{t^c}$ and $\partial_{u_{cd}}$.  
 If $\{ \chi_i \}$ is a basis of $L$ and $\{ \chi^i \}$ its dual basis, then $S^2 L^*$ is spanned by $\chi^i\chi^j$ (corresponding to $\partial_{u_{ij}}$).  Let $l(v) = \chi_0 - t^a \chi_a \in L$.  Then $\fr^\perp(v) = \fann(\fr(v)) = \tspan\{ l^2 \} \subset S^2 L$, i.e.\ $\cF$ is of parabolic type.  (If $\cF = \{ \sfF(x^i,u,u_i,u_{ij}) = 0 \}$, then $\rnk(\frac{\partial \sfF}{\partial u_{ij}})=1$ everywhere.)
 
 We have a distinguished $l^\perp \subset L^*$ and a corresponding {\em first-order covariant system} $\cN \subset \cD^{-2}$ spanned by $\cD$ together with $t^a \partial_{u_0} + \partial_{u_a}$.  The (graded Lie algebra) automorphism group $A(\fs)$ of $\fs$ distinguishes a subspace $\tspan\{ \phi(l) : \phi \in A(\fs) \} \subset \fs_{-1}$ and we let $\cM \subset \cD$ be the corresponding distribution, called the {\em Monge characteristic system}.  (See \cite[\S 7.3]{Yam1982} for more details.)  Explicitly, $\cN$ and $\cM$ are spanned by
 \begin{align*}
 \cN: &\qquad \partial_{x^0} + u_0 \partial_u - \frac{1}{2} \fC(t^3) \partial_{u_0}, \quad
 \partial_{x^a} + u_a \partial_u - \frac{3}{2} \fC_a(t^2) \partial_{u_0}, \quad
\partial_{t^a}, \quad \partial_{u_{ab}}, \quad t^a \partial_{u_0} + \partial_{u_a};\\
 \cM: &\qquad \partial_{x^0} + u_0 \partial_u - t^a ( \partial_{x^a} + u_a \partial_u) - \frac{\fC(t^3)}{2} \partial_{u_0} - \frac{3\fC_a(t^2)}{2}  \partial_{u_a}, \quad \partial_{u_{ab}}.
 \end{align*}
 Since $\cM = \Ch(\cN)$, then $\cM$ is completely integrable, i.e.\ $\cF$ is of {\em Goursat} type.
 
 The vertical bundle for $\cF \to M$ is $\Ch(\widetilde\cC) = \tspan\{ \partial_{t^a}, \partial_{u_{ab}} \}$.  We have $\Ch(\cN) \cap \Ch(\widetilde\cC) = \tspan\{ \partial_{u_{ab}} \}$, and the projection of $\Ch(\cN)$ to $M$ recovers the sub-adjoint variety field $\cV \subset \bbP(\cC)$.  Since $\cN$ and $\Ch(\cN)$ are covariant for $\cD$, this confirms that all symmetries of $(\cF,\cD)$ are inherited by $(M,\cC,\cV)$. 
  
 \subsection{Degenerate cases}
 \label{S:degenerate}
 
 In this section, we treat the exceptional type $A$ and $C$ cases.
 
 \subsubsection{Type A}
 \label{S:typeA}
 
 For $G = A_{n+1}$, $A_{n+1} / P_{1,n+1} \cong G^{\ad} \inj \bbP(\fg)$ and $\fg_{-1}$ is reducible for $\fg_0 \cong \bbC^2 \times \fsl_n$.  Here, the corresponding geometric structure is a {\em Legendrian contact structure}, i.e.\ $(M^{2n+1},\cC)$ with a decomposition $\cC = E \op F$ into complementary Legendrian subspaces. When $F$ is integrable, we can introduce coordinates $(x^i,u,u_i)$, $1 \leq i \leq n$, so that $\cC = \fann(du - u_i dx^i)$ and $F = \tspan\{ \partial_{u_k} \}$ (see \cite{DMT2016}).  Now $E$ specifies a section of $M^{(1)} \to M$, so such structures can be regarded as complete systems of 2nd order PDE $u_{ij} = f_{ij}(x^k,u,u_k)$ up to {\em point} transformations, i.e.\ contact transformations preserving $\tspan\{ \partial_{u_k} \}$.

 \begin{prop} The PDE system $u_{ij} = 0$, $1 \leq i,j \leq n$ has point symmetry algebra $\fg = A_{n+1}$.  Its generating functions are:
 \begin{align*}
 \begin{array}{|c|c|c|c|c|c|c|} \hline
 \fg_{-2} & \multicolumn{2}{|c|}{\fg_{-1}} & \fg_0 & \multicolumn{2}{|c|}{\fg_1} & \fg_2 \\ \hline
 \fg_{-1,-1} & \fg_{-1,0} & \fg_{0,-1} & \fg_{0,0} & \fg_{1,0} & \fg_{0,1} & \fg_{1,1} \\ \hline
 1 & x^i & u_i & \begin{array}{c} \sfZ = 2u - x^i u_i \\ x^i u_j \end{array} & 
 u u_i & x^i (u - x^j u_j) & u (u - x^j u_j) \\ \hline
 \end{array}
 \end{align*}
 Here, $\sfZ = \sfZ_1 + \sfZ_{n+1}$, where the bi-grading element $(\sfZ_1,\sfZ_{n+1}) = ( u, u - x^i u_i)$ acts as indicated above.
 \end{prop}
 
 The top slot is a special case of \eqref{E:top-slot} when $\fC = 0$.  (Here, $x^0, u_0$ are no longer distinguished.)\\
 
 Harmonic curvature obstructs flatness and corresponds to $H^2_+(\fg_-,\fg)$.  We summarize its components:
 \begin{itemize}
 \item $n=1$: Relative invariants $I_1, I_2$ ({\em Tresse invariants}).
 \item $n \geq 2$: Two torsions $\tau_E,\tau_F$ (obstructing the integrability of $E,F$) and a curvature $\cW$.  For semi-integrable structures, $u_{ij} = f_{ij}(x^k,u,u_k)$, $\cW$ has components $ \cW_{ij}^{kl} = \operatorname{trfr}\left( \frac{\partial^2 f_{ij}}{\partial u_k \partial u_l} \right)$; see \cite{DMT2016}.

 \end{itemize}
 
 \subsubsection{Type C}
 \label{S:TypeC}
 
 For $G = C_{n+1}$, $C_{n+1} / P_1 \cong G^{\ad} \inj \bbP(\fg)$ and $G_0^{\Ss} \cong \Sp(\fg_{-1})$ acts transitively on $\tLG(\fg_{-1})$, so instead examine the stabilizer $P_{1,n+1}$  of a Lagrangian subspace.  This induces a $|3|$-grading $\fg = \fg_{-3} \op... \op \fg_3$ with $\fg_0 \cong \bbC^2 \times A_{n-1}$ and $\fg_-$ having the same commutator relations as \eqref{E:J2-symbol} for $(J^2(\bbC^n,\bbC),\cC^{(1)})$, where $\dim(L) = n:= \ell-1$.  In $\fg_{-1} = L \op S^2 L^*$, {\em both} $L$ and $S^2L^*$ are distinguished subspaces under the $G_0$-action.

 By \cite[Cor.6.6]{Yam1983}, any $(N,\cD)$ with symbol algebra modelled on $\fg_-$ (and $\cD$ modelled on $\fg_{-1}$) is locally isomorphic to $(J^2(\bbC^n,\bbC), \cC^{(1)})$.  The structure underlying a regular, normal $(C_{n+1}, P_{1,n+1})$ geometry is a further choice of subbundle $E \subset \cD$ complementary to $V = \Ch(\cD^{-2}) \subset \cD$.  (These correspond to $L$ and $S^2 L^*$ in $\fg_{-1}$.)  Locally, $\cD \cong \cC^{(1)}$ is given by \eqref{E:J2-cansys} and $V \cong \tspan\{ \partial_{u_{ij}} \}$.  Hence, 
 \[
 E = \tspan\{ \widetilde\partial_{x^i} := \partial_{x^i} + u_i \partial_u + u_{ij} \partial_{u_j} + f_{ijk} \partial_{u_{jk}} \},
 \]
 for some functions $f_{ijk}(x^l,u,u_l,u_{l m})$ (symmetric in $i,j,k$).
 Equivalently, the geometric structure corresponds to the {\em contact} geometry of a complete system of 3rd order PDE 
 \begin{align} \label{E:3rdOrder}
 u_{ijk} = f_{ijk}(x^l,u,u_l,u_{lm}), \qquad 1 \leq i,j,k,l,m \leq n.
 \end{align}
 The $n=1$ case is the contact geometry of a 3rd order ODE.
 The flat model is:

  \begin{prop}
 The PDE system $u_{ijk} = 0$, $1 \leq i,j,k \leq n$ has contact symmetry algebra $\fg = C_{n+1}$ and isotropy $P_{1,n+1}$ on the second jet space $(x^i,u,u_i,u_{ij})$.  The generating functions are:
 \begin{align*}
 \begin{array}{|c|c|c|c|c|c|c|} \hline
 \fg_{-3} & \fg_{-2} & \fg_{-1} & \fg_0 & \fg_1 & \fg_2 & \fg_3\\ \hline
 1 & x^i & \begin{array}{c} u_i\\ x^i x^j \end{array} & \begin{array}{c} u \\ x^i u_j \end{array} &
 \begin{array}{c} x^i (x^j u_j - 2 u) \\ u_i u_j \end{array} & 
 u_i (x^j u_j - 2 u) & (x^j u_j - 2u)^2\\ \hline
 \end{array}
 \end{align*}
 \end{prop}
   
 Harmonic curvature $\kappa_H$ can be computed in a similar fashion as in \S \ref{S:Kh}.  On $\cC^{(1)}$, define a partial connection $\nabla$ such that $\nabla (\widetilde\partial_{x^i}) = 0$ and $\nabla(\partial_{u_{l m}}) = 0$, which clearly preserves the splitting $\cC^{(1)} = E \op V$.   Restricting to $n \geq 2$,  $H^2_+(\fg_-,\fg)$ decomposes into two $\fg_0$-irreps, both of homogeneity $+1$ and comprised of torsion.  (If $n=1$, a curvature appears, so this case is different.)  Hence, $\kappa_H$ is comprised of two components $\tau_E$ and $\tau_{EV}$ of torsion $T^\nabla$ and these are obtained as follows:
 \begin{enumerate}
 \item[(i)] In $\fg_-$, $V$ is modelled on $S^2 L^*$.  The Cartan product $\bigwedge^2 L^* \odot S^2 L^*$ is the kernel of the skew-symmetrization map on the first three factors of $\bigwedge^2 L^* \otimes S^2 L^*$. We calculate
 \[
 -T^\nabla(\widetilde\partial_{x^l},\widetilde\partial_{x^m}) = [ \widetilde\partial_{x^l},  \widetilde\partial_{x^m} ] =  \left( \widetilde\partial_{x^l} (f_{mjk}) - \widetilde\partial_{x^m} (f_{l jk}) \right) \partial_{u_{jk}} =: T_{l m jk} \partial_{u_{jk}}
 \]
 Thus, $\tau_E \in \Gamma(\bigwedge^2 E^* \odot V)$ has components $(\tau_E)_{l mjk} = T_{l m jk} - T_{[l m j]k}$.
 
 \item[(ii)] Let $(E \otimes V)_0$ be the kernel of the Levi-bracket restricted to $E \otimes V$.  Then $\tau_{EV} \in  \operatorname{trfr}((E \otimes V)_0)^* \otimes V)$, where all traces have been removed.  We calculate
 \[
 T^\nabla(\widetilde\partial_{x^i},\partial_{u_{l m}}) = [\partial_{u_{l m}}, \widetilde\partial_{x^i}] = \delta^l_i \partial_{u_m} + \delta^m_i \partial_{u_l} + \frac{\partial f_{ijk}}{\partial u_{l m} } \partial_{u_{jk}}.
 \]
 Now we need to remove all traces.  Let $R^{l m}_{ijk} = \frac{\partial f_{ijk}}{\partial u_{l m}}$, $S^m_{jk} = R^{rm}_{rjk}$ and $T_k = S^r_{rk}$.  Define
 \[
 (\tau_{EV})_{ijk}^{l m} := \operatorname{trfr}\left( \frac{\partial f_{ijk}}{\partial u_{l m} }\right) =   R^{l m}_{ijk} - \frac{6}{n+3} \delta^{(l}_{(i} S^{m)}_{jk)} + \frac{6}{(n+2)(n+3)} \delta^{(l}_{(i} \delta^{m)}_{j} T^{\phantom{\ast}}_{k)}.
 \]
 \end{enumerate}

 Since $\kappa_H$ completely obstructs flatness, we have:
 
 \begin{theorem} Let $n \geq 2$. The complete 3rd order PDE system \eqref{E:3rdOrder} is contact equivalent to the flat model $u_{ijk} = 0$, $1 \leq i,j,k \leq n$ if and only if $\tau_E = 0$ and $\tau_{EV} = 0$.
 \end{theorem}
 
 In Definition \ref{D:G-contact}, we omitted $C_{n+1}$-contact structures (contact projective structures; $C_{n+1} / P_1$ geometries) since these are instead encoded via a class of contact connections.  Such structures are equivalent, via \v{C}ap's theory of correspondence and twistor spaces \cite{Cap2005}, to complete systems of 3rd order PDE (i.e.\ $C_{n+1} / P_{1,n+1}$ geometries) satisfying $\tau_{EV} = 0$.
 
 \section{Non-flat structures}
 \label{S:non-flat}
 
 \subsection{$G_2$-contact structures}
 \label{S:G2P2}
 
 Following \S \ref{S:Kh}, we exhibit a formula for $\kappa_H$ for $G_2$-contact structures, and then we establish some symmetry classification results.  We follow notation introduced in Example \ref{EX:G2-1}.  Here, $V = \fg_{-1} = S^3 \bbC^2$, $\fg_0 \cong \fgl_2$, the grading element is $\sfZ = -\frac{1}{3} \sfI$, and we find (via Kostant) that:
 \begin{itemize}
 \item $H^2_+(\fg_-,\fg)$ arises as a $\fg_0$-irreducible component of $\bigwedge^2_0 V^* \otimes V$.  (In particular, $\sfZ$ acts as $+1$.)
 \item $H^2_+(\fg_-,\fg) \cong S^7 \bbC^2 \otimes (\bigwedge^2 (\bbC^2)^*)^5$ as $\fg_0$-modules, i.e.\ (weighted) binary septics.
 \end{itemize}
 Let $\Gamma_k$ denote the $\fsl_2$-module $S^k \bbC^2$. 
 Then $\bigwedge^2_0 V^* \otimes V \cong \Gamma_4 \otimes \Gamma_3 \cong \Gamma_7 \op \Gamma_5 \op \Gamma_3 \op \Gamma_1$.  The $\fsl_2$-equivariant projection $\Gamma_4 \otimes \Gamma_3 \to \Gamma_7$ is simply multiplication of polynomials, i.e.\ $f \otimes g \mapsto fg$.
 
 Let us exhibit $\bigwedge^2_0 V^* \cong \Gamma_4$ explicitly.  Consider the $\fsl_2$-basis $(\sfe_0,\sfe_1,\sfe_2,\sfe_3) = (\sfr^3, 3\sfr^2 \sfs, 3\sfr\sfs^2, \sfs^3)$ of $V$.  Let $\omega^0, \omega^1, \omega^2, \omega^3$ be its dual basis.  Then $\eta$ from \eqref{E:G2-eta} is a multiple of $\omega^0 \wedge \omega^3 - 3 \omega^1 \wedge \omega^2$.  Hooking with the latter yields an $\fsl_2$-isomorphism $V \cong V^*$, which identifies $(\omega^0,\omega^1, \omega^2, \omega^3) = (-\sfs^3, \sfr\sfs^2, -\sfr^2\sfs, \sfr^3)$.  Hence,
 \begin{align} \label{E:sl2-id}
 \begin{array}{|lccccc|} \hline
 \mbox{Element of $\bigwedge^2_0 V^*$:} & \omega^2 \wedge \omega^3 & \omega^3 \wedge \omega^1 & \frac{1}{2} (\omega^0 \wedge \omega^3 + 3 \omega^1 \wedge \omega^2) & \omega^2 \wedge \omega^0 & \omega^0 \wedge \omega^1\\
 \mbox{Element of $\bigwedge^2_0 V$:} & \sfr^3 \wedge \sfr^2\sfs & \sfr^3 \wedge \sfr\sfs^2 & \frac{1}{2}(\sfr^3 \wedge \sfs^3 + 3\sfr^2\sfs \wedge \sfr\sfs^2) & \sfr^2\sfs \wedge \sfs^3 & \sfr\sfs^2 \wedge \sfs^3\\
 \mbox{Element of $\Gamma_4$:} & \sfr^4 & 2\sfr^3\sfs & 3\sfr^2\sfs^2 & 2\sfr\sfs^3 & \sfs^4\\ \hline
 \end{array}
 \end{align}
 Next, note that a CS-basis of $(V,[\eta])$ is given by $(\sfr^3, -3\sfr^2 \sfs, -6\sfs^3, -6\sfr\sfs^2 )$.  We can pointwise identify this with a given CS-framing $(\bX_0,\bX_1,\bU^0,\bU^1)$ defining a $G_2$-contact structure.
 
 \begin{theorem} Consider the $G_2$-contact structure on $(M^5,\cC)$ for a CS-framing $\bX_0,\bX_1,\bU^0,\bU^1$ of $\cC$.  Define
 \begin{align} \label{E:sl2-frame}
 \bE_0 = \bX_0, \quad \bE_1 = -\bX_1, \quad \bE_2 = -\frac{1}{2} \bU^1, \quad \bE_3 = -\frac{1}{6} \bU^0.
 \end{align}
 Given $\bY \in \Gamma(\cC)$, define $\rho(\bY) = \sfr^3 \rho_0(\bY) + 3\sfr^2\sfs \rho_1(\bY) + 3\sfr\sfs^2 \rho_2(\bY) + \sfs^3 \rho_3(\bY)$, where $\{ \rho_i(\bY) \}$ are components with respect to $\{ \bE_i \}$.  Then $\kappa_H$ is (up to a constant) the tensor product of $0 \neq \vol^5 \in (\bigwedge^2 (\bbC^2)^*)^5$ with
 \begin{align} \label{E:G2P2-Kh}
 \sfr^4 \rho([\bE_2,\bE_3])  + 2\sfr^3 \sfs \rho([\bE_3,\bE_1])  + \sfr^2\sfs^2 \rho\left(3[\bE_0,\bE_3] + [\bE_1,\bE_2]\right) + 2\sfr\sfs^3\rho([\bE_2,\bE_0]) + \sfs^4\rho([\bE_0,\bE_1]).
 \end{align}
 The $G_2$-contact structure is flat iff $\kappa_H = 0$.
 \end{theorem}
 
 \begin{proof} The framing \eqref{E:sl2-frame} corresponds to the $\fsl_2$-basis $(\sfe_0,\sfe_1,\sfe_2,\sfe_3) = (\sfr^3, 3\sfr^2 \sfs, 3\sfr\sfs^2, \sfs^3)$.  As in \S \ref{S:Kh}, pick $\nabla$ that leaves the CS-framing, hence \eqref{E:sl2-frame}, parallel.  The basis for $\bigwedge^2_0 V^*$ in \eqref{E:sl2-id} has dual basis $\sfe_2 \wedge \sfe_3, \, \sfe_3 \wedge \sfe_1, \, \frac{1}{3} (3\sfe_0 \wedge \sfe_3 + \sfe_1 \wedge \sfe_2), \, \sfe_2 \wedge \sfe_0,\, \sfe_0 \wedge \sfe_1$.  Evaluate $T^\nabla$ on corresponding bivectors (formed from the framing \eqref{E:sl2-frame}) and find their components via $\rho$.  Multiply with corresponding elements of $\Gamma_4$ to obtain \eqref{E:G2P2-Kh}.
 \end{proof}
 
  One can naturally classify non-flat $G_2$-contact structures according to the root type of $\kappa_H$.  Some homogeneous examples are given in Table \ref{F:G2P2-hom}, with $\cV$, $[\sfQ]$, $\cE$, $\cF$ determined by \eqref{E:G-SA-var}--\eqref{E:G-Goursat}.  Note that to write $\cE$ (or $\cF$) in standard jet-coordinates, we identify the CS-element $g$ specifying the frame change from the standard CS-framing \eqref{E:std-CS} (corresponding to $u_{xx} = \frac{t^3}{3}, \, u_{xy} = \frac{t^2}{2},\, u_{yy} = t$) to the given one, and then use \eqref{E:Mobius}.  These PDE satisfy additional compatibility conditions.  A case analysis shows that the solution space generally depends on 3 arbitrary constants.  However, our type [7] PDE example is {\em inconsistent}, while our type [2,2,1,1,1] PDE example has general solution $u(x,y) = c_1 + c_2 x + c_3 y$ or an {\em arbitrary} function $f(y)$.  For all these examples, $(\cE,\cD)$ does not admit any Cauchy characteristics.

 \begin{footnotesize}
 \begin{table}[h]
 \[
 \begin{array}{|c|c|c|c|c|c|} \hline
 \mbox{Root type} & \bX_0 & \bX_1 & \bU^0 & \bU^1 & \begin{tabular}{c} $G_2$-contact structure\\ symmetries \end{tabular}\\ \hline\hline
 [7] & \partial_x + p\partial_u + y\partial_p & \partial_y + q\partial_u & \partial_p & \partial_q & \begin{array}{c} 1,\, x,\, y,\, p,\, x^3-3y^2-6qx, \\ 2q-x^2, \,7u-2px-3qy\end{array}\\ \hline
 [6,1] & \partial_x + p\partial_u + q\partial_p & \partial_y + q\partial_u & \partial_p & \partial_q & \begin{array}{c} 1,\, x,\, p,\, q,\, x^2+2y,\, 5u - px - 2qy \end{array}\\ {}
 [5,2] & \partial_x + p\partial_u + p\partial_p & \partial_y + q\partial_u & \partial_p & \partial_q & 1, \, y,\, p,\, q,\, e^x, \, 3u - qy \\ {}
 [4,3] & \partial_x + p\partial_u & \partial_y + q\partial_u + q\partial_q & \partial_p & \partial_q & 1,\, x,\, p,\, q,\, e^y,\, u+px \\ \hline
 [5,1,1] & \partial_x + p\partial_u + (p+q) \partial_p & \partial_y+q\partial_u & \partial_p & \partial_q & 1,\, p,\, q,\, e^x,\, y-x\\ {}
 [4,2,1] & \partial_x+p\partial_u+p\partial_p & \partial_y + q\partial_u -q\partial_q & \partial_p & \partial_q & 1,\, p,\, q,\, e^{-y},\, e^x\\ {}
 [3,3,1] & \partial_x + p\partial_u & \partial_y + q\partial_u+(p+q)\partial_q & \partial_p & \partial_q & 1,\, p,\, q,\, e^y,\, y-x\\ {}
 [3,2,2] & (q+\frac{x}{6})(\partial_x + p\partial_u) & y(\partial_y + q\partial_u) & y\partial_p & (q+\frac{x}{6})\partial_q & 1,\, x,\, u - px,\, u-qy,\, y + 6p \\ \hline
 [4,1,1,1] & \partial_x+p\partial_u & \partial_y+q\partial_u+(x+q)\partial_q & \partial_p & \partial_q & 1,\,x,\,q,\,p+y,\,e^y\\ {}
 [3,2,1,1] & (x+q)(\partial_x + p\partial_u) & y(\partial_y + q\partial_u) & y\partial_p & (x+q)\partial_q & 1,\, x,\, y+p,\, u-px,\, u-qy\\ {}
 [2,2,2,1] & (q-\frac{x}{6})(\partial_x + p\partial_u) & (p+\frac{y}{2})(\partial_y + q\partial_u) & (p+\frac{y}{2})\partial_p & (q-\frac{x}{6})\partial_q & 1,\, x+2q,\, y-6p,\, u-px,\, u-qy\\ \hline
 [3,1,1,1,1] & \partial_x + p\partial_u & \partial_y + q\partial_u+(x+p+q)\partial_q & \partial_p & \partial_q & 1,\,q,\,e^y,\,x-y,\,p+y\\ {}
 [2,2,1,1,1] & x(\partial_x + p\partial_u) & p(\partial_y+q\partial_u) & \frac{1}{x} \partial_p & \frac{1}{p} \partial_q & 1,\,y,\,q,\,4px-3u,\,4qy-3u \\ \hline
 [2,1,1,1,1,1] & q(\partial_x+p\partial_u)+p\partial_p & p(\partial_y+q\partial_u) & p\partial_p & q\partial_q & 1,\, p,\, q,\, qy-u,\, px-u\\ \hline
 \end{array}
 \]
 \[
 \begin{array}{|c|l|c|} \hline
 \mbox{Root type} & \mbox{$G_2$-contact structure PDE $\cE$ associated to the models above} \\ \hline\hline
 [7] & 
 	u_{xx} = \frac{1}{3} (u_{yy})^3 + y, \quad
	u_{xy} = \frac{1}{2} (u_{yy})^2 \\ \hline
 [6,1] & 
 	u_{xx} = \frac{1}{3} (u_{yy})^3 + u_y, \quad 
 	u_{xy} = \frac{1}{2} (u_{yy})^2 \\ {}
 [5,2] & 
 	u_{xx} = \frac{1}{3} (u_{yy})^3 + u_x, \quad 
	u_{xy} = \frac{1}{2} (u_{yy})^2 \\ {}
 [4,3] & 
 	u_{xx} = \frac{1}{3} (u_{yy}-u_y)^3, \quad 
	u_{xy} = \frac{1}{2} (u_{yy}-u_y)^2 \\ \hline
 [5,1,1] & 
 	u_{xx} = \frac{1}{3} (u_{yy})^3 + u_x + u_y,\quad 
	u_{xy} = \frac{1}{2} (u_{yy})^2 \\ {}
 [4,2,1] & 
 	u_{xx} = \frac{1}{3} (u_{yy}+u_y)^3 + u_x,\quad 
	u_{xy} = \frac{1}{2} (u_{yy}+u_y)^2 \\ {}
 [3,3,1] & 
 	u_{xx} = \frac{1}{3} (u_{yy}-u_x-u_y)^3,\quad 
	u_{xy} = \frac{1}{2} (u_{yy}-u_x-u_y)^2 \\ {}
 [3,2,2] & 
 	u_{xx} = \frac{(u_{yy})^3 y^4}{3(u_y + \frac{x}{6})^4},\quad 
	u_{xy} = \frac{(u_{yy})^2 y^2}{2(u_y + \frac{x}{6})^2} \\ \hline
 [4,1,1,1] & 
 	u_{xx} = \frac{1}{3} (u_{yy}-x-u_y)^3, \quad
	u_{xy} = \frac{1}{2} (u_{yy}-x-u_y)^2 \\ {}
 [3,2,1,1] & 
 	u_{xx} = \frac{(u_{yy})^3 y^4}{3(u_y + x)^4},\quad
	u_{xy} = \frac{(u_{yy})^2 y^2}{2(u_y + x)^2} \\ {}
 [2,2,2,1] & 
 	u_{xx} = \frac{1}{3} (u_{yy})^3 \left( \frac{u_x + \frac{y}{2}}{u_y - \frac{x}{6}}\right)^4,\quad
	u_{xy} = \frac{1}{2} (u_{yy})^2 \left( \frac{u_x + \frac{y}{2}}{u_y - \frac{x}{6}}\right)^2 \\ \hline
 [3,1,1,1,1] & 
 	u_{xx} = \frac{1}{3} (u_{yy} - x - u_x - u_y)^3,\quad
	u_{xy} = \frac{1}{2} (u_{yy} - x - u_x - u_y)^2 \\ {}
 [2,2,1,1,1] & 
 	u_{xx} = \frac{1}{3} (u_{yy})^3 \frac{(u_x)^6}{x^2} ,\quad
	u_{xy} = \frac{1}{2} (u_{yy})^2 \frac{(u_x)^3}{x} \\ \hline
 [2,1,1,1,1,1] & 
 	u_{xx} = \frac{1}{3} (u_{yy})^3 \frac{(u_x)^4}{(u_y)^4} + \frac{u_x}{u_y},\quad
	u_{xy} = \frac{1}{2} (u_{yy})^2 \frac{(u_x)^2}{(u_y)^2}\\ \hline
 \end{array}
 \]
  \caption{Some homogeneous $G_2$-contact structures on $(J^1(\bbC^2,\bbC), [du - pdx - q dy])$}
 \label{F:G2P2-hom}
 \end{table}
 \end{footnotesize}

 \begin{prop}
 Consider a non-flat $G_2$-contact structure on $(M^5,\cC)$ with harmonic curvature $\kappa_H$ and symmetry algebra $\cS$.  Fix any $m \in M$.  Then:
 \begin{enumerate}
 \item[(i)] If $\kappa_H(m)$ has root type $[7]$, then $\dim(\cS) \leq 7$.
 \item[(ii)] If $\kappa_H(m)$ has root type $[6,1], [5,2]$, or $[4,3]$, then $\dim(\cS) \leq 6$.
 \item[(iii)] If $\kappa_H(m)$ has $\geq 3$ distinct roots, then $\dim(\cS) \leq 5$.
 \end{enumerate}
 For all root types except possibly $[1,1,1,1,1,1,1]$, these upper bounds are sharp.
 \end{prop}
 
 \begin{proof} We have $\dim(\cS) \leq 5 + \dim(\fann(\kappa_H(m)))$.  This follows immediately from \cite[Thm.3.3]{KT2016} (Tanaka prolongation dimension gives a pointwise upper bound) and \cite[Cor.3.4.8]{KT2014} ($G_2 / P_2$ is prolongation-rigid).  For (i) and (ii), take $\tGL_2$-representative elements $\sfr^{7-a}\sfs^a \otimes \vol^5$, which have annihilators in $\fg_0$ spanned by:
 \[
 a=7: \,\, 7\sfI - 3\sfH, \, \sfY; \qquad
 a=6: \,\, 5\sfI - 3\sfH; \qquad
 a=5: \,\, \sfI - \sfH; \qquad
 a=4: \,\, \sfI - 3\sfH.
 \]
 This proves (i) and (ii).  For (iii), the annihilator is trivial.  The final statement follows from Table \ref{F:G2P2-hom}.
 \end{proof}
 
 We do not know any {\em homogeneous} $G_2$-contact structures of root type $[1,1,1,1,1,1,1]$.  The CS-framing
 \[
 \bX_0 = \partial_x + p\partial_u, \quad
 \bX_1 = \partial_y + q\partial_u + u\partial_q, \quad
 \bU^0 = \partial_p, \quad
 \bU^1 = (1+up) \partial_q + p(\partial_y + q\partial_u)
 \]
 determines a type $[1,1,1,1,1,1,1]$ structure with 4-dimensional symmetry algebra spanned by $p, q, e^{-y}, e^y$.

 \subsection{Submaximally symmetric $G$-contact structures}
 \label{S:submax-sym}
 
 Let $G \neq A_\ell,C_\ell$ with associated Jordan algebra $W$ (Table \ref{F:sym}) and basis $\{ \sfw_a \}$.  Fix $1 \leq c \leq \dim(W)$ and take the $G$-contact structure for the CS-framing
 \begin{align} \label{E:curved-G-contact}
 \bX_0 = \partial_{x^0} + u_0 \partial_u + x^c \partial_{u_0}, \quad
 \bX_a = \partial_{x^a} + u_a \partial_u, \quad
 \bU^0 = \partial_{u_0}, \quad
 \bU^a = \partial_{u_a}.
 \end{align}
 As in \S \ref{S:G2P2}, we can find the corresponding $\cE$ and $\cF$ in standard jet-coordinates:
 \begin{align}
  \cE: & \quad (u_{ij}) 
 = \displaystyle\left( \begin{array}{c|c} u_{00} & u_{0b} \\ \hline u_{a0} & u_{ab}\end{array} \right) 
 = \left( \begin{array}{c|c} \fC(t^3) + x^c & \frac{3}{2} \fC_b(t^2)\\[0.05in] \hline \frac{3}{2} \fC_a(t^2) & 3\fC_{ab}(t)  \end{array} \right), \label{E:sm-EV}\\
  \cF: & \quad \left\{ \begin{array}{l} u_{00} = t^a t^b u_{ab} - 2\fC(t^3) + x^c\\ u_{a0} = t^b u_{ab} - \frac{3}{2} \fC_a(t^2)\end{array} \right., \label{E:sm-FV}
 \end{align}
 which are also the first and second-order envelopes determined by the parametrization
 \[
  \fP_t : \quad  u_{00} - 2 t^a u_{a0} + t^a t^b u_{ab} = \fC(t^3) + x^c.
 \]
 
 \begin{prop} \label{P:sym-non-flat}
 The symmetry algebra of the $G$-contact structure \eqref{E:curved-G-contact} is spanned by
 \begin{align}
 & 1, \quad x^i, \quad u_i - \delta_i{}^c \frac{(x^0)^2}{2}, \quad 
 7u - 2 x^0 u_0 - 3 x^a u_a, \quad
 x^0 u_a + \frac{3}{2} \fC_a(X^2) - \frac{1}{6} \delta_a{}^c (x^0)^3, \label{E:curved-sym-1}\\
 & 7 A^a{}_b x^b u_a - k \left( 3x^0 u_0 + x^a u_a \right),  \qbox{where}\quad 0 = A^a{}_{(b} \fC_{de)a} \qbox{and} A^c{}_b = k \delta^c{}_b, \label{E:curved-sym-2}
 \end{align}
 and $0 \leq i \leq n-1$ and $1 \leq a,b \leq n-1$.
 \end{prop}
 
 \begin{proof}
 This calculation is similar to that for Theorem \ref{T:sym-flat}.  We have $\bV := \bV(1,t)$ as in \eqref{E:G-SA-var} and $\widehat\cV = \widehat{T}_{[\bV]} \cV$ as in \eqref{E:G-involutive}.  
 Let $\rho^i := dx^i([\bV,\bS_f]) = -\bV(f_{u_i})$ and $\mu_i := du_i([\bV,\bS_f]) = \bV\left( \frac{df}{dx^i} \right) + \delta_i{}^0 f_{u_c}$.  Then
 \[
 [\bV,\bS_f] = \rho^0 \bX_0 + \rho^a \bX_a + (\mu_0 - \rho^0 x^c )\bU^0 + \mu_a \bU^a.
 \]
  Using \eqref{E:G-involutive}, the condition $[\bV,\bS_f] \in \widehat{T}_{[\bV]} \cV$, expressed in the CS-framing $\{ \bX_i, \bU^i \}$ is
 \begin{align}
 0 &= \mu_0 - \rho^0 (x^c + \fC(t^3)) - \frac{3}{2} \rho^a \fC_a(t^2) 
  = \bV\left( \frac{df}{dx^0} \right) + f_{u_c} + \bV(f_{u_0}) (x^c + \fC(t^3)) - \frac{3}{2} \bV(f_{u_a}) \fC_a(t^2) \label{E:curved-1}\\
 0 &= \mu_a - \frac{3}{2} \rho^0 \fC_a(t^2) - 3 \rho^b \fC_{ab}(t)
 = \bV\left( \frac{df}{dx^a} \right) + \frac{3}{2} \bV(f_{u_0}) \fC_a(t^2) - 3 \bV(f_{u_b}) \fC_{ab}(t) \label{E:curved-2}
 \end{align}
 The $t$-degree 6 and 5 components of \eqref{E:curved-1} imply $f_{u_0 u_0} = f_{u_0 u_i} = 0$, so $f = g(x^i,u,u_a) + h(x^i,u) u_0$. Setting $u_0 = 0$ in \eqref{E:curved-1}[$t$-degree 1] implies
 \begin{align} \label{E:curved-3}
 0 = (h_{x^a} + u_a h_u) x^c, \qquad 0 = g_{x^a x^0} + u_a g_{u x^0}.
 \end{align}
 The first equation in \eqref{E:curved-3} implies $h = h(x^0)$.  Furthermore, the equations \eqref{E:curved-1}[$(t,u_0)$-degree $(0,2)$], \eqref{E:curved-2}[$(t,u_0)$-degree $(0,1)$], \eqref{E:curved-1}[$t$-degree 4], and \eqref{E:curved-2}[$(t,u_0)$-degree $(1,1)$] imply
 \begin{align} \label{E:curved-4}
  g_{ux^a} = 0, \qquad g_{uu} = 0, \qquad g_{u_b u_e} = 0, \qquad g_{u u_b} = 0.
 \end{align}
 The third of these follows from \eqref{E:curved-1}[$t$-degree 4] and \eqref{E:dual-cubic-id}: $0 = -t^a \bX_a(f_{u_0}) \fC(t^3) + \frac{9}{4} \fC_a(t^2) \fC_b(t^2) f_{u_a u_b} = \frac{9}{4} (3(\fC^*)^{abe} \bX_a(f_{u_0}) + f_{u_b u_e}) \fC_b(t^2) \fC_e(t^2)$.  Since $\fC_b(t^2)$ is arbitrary and $f_{u_a u_0} = f_{u u_0} = 0$, then $f_{u_b u_e} = 0$.
 
 Now combine \eqref{E:curved-3} and \eqref{E:curved-4}, and after a straightforward calculation, we obtain
 \begin{align} \label{E:curved-5}
 h = h(x^0), \quad g = G^a(x^i) u_a + \alpha(x^a) + \beta(x^0) + \gamma(x^0) u.
 \end{align}
 Using \eqref{E:curved-2}[$t$-degree 1] and the coefficients of $u$, $u_a$, and $u_0$ in \eqref{E:curved-1}[$t$-degree 0], these must satisfy
 \begin{align}
 & 0 = (G^b)_{x^a x^0} + \delta_a{}^b \gamma_{x^0}, \qquad (G^b)_{x^e x^a} = 0, \qquad \alpha_{x^e x^a} - 3 (G^b)_{x^0} \fC_{abe} = 0, \label{E:curved-6} \\
 &\gamma_{x^0 x^0} = 0, \qquad (G^a)_{x^0 x^0} = 0, \qquad h_{x^0 x^0} + 2\gamma_{x^0} = 0, \label{E:curved-7}\\
 & \beta_{x^0 x^0} + x^c (2h_{x^0} + \gamma) + G^c = 0. \label{E:curved-8}
 \end{align}
 (The first equation in \eqref{E:curved-6} is a residual equation from \eqref{E:curved-3}.)   
  Differentiating \eqref{E:curved-8} with respect to $x^0,x^c$ yields $0 = 2 h_{x^0 x^0} + \gamma_{x^0} + (G^c)_{x^0 x^c} = -4\gamma_{x^0}$ (no sum on $c$).  Hence, we must have
 \begin{align}
 & \gamma = \gamma_0, \qquad h = \widetilde{k} x^0 + h_0, \qquad 
  G^a = \widetilde{A}^a{}_b x^b + B^a x^0 + C^a, \qquad 
  \widetilde{A}^c{}_b = -\delta^c{}_b (2\widetilde{k} + \gamma_0)
\\
 & \alpha = \frac{3}{2} B^b \fC_b(X^2) + \alpha_b x^b + \alpha_0, \qquad \beta = -\frac{1}{6} B^c (x^0)^3 - \frac{1}{2} C^c (x^0)^2 + \beta_1 x^0.
 \end{align}
  Define $\widetilde{A}^a{}_b = A^a{}_b + \frac{1}{3} \delta^a{}_b (\widetilde{k} - \gamma_0)$.  Then $A^c{}_b = (-\frac{7}{3} \widetilde{k} - \frac{2}{3} \gamma_0) \delta^c{}_b$ and \eqref{E:curved-1}[$t$-degree 3] yields $0 = A^a{}_{(b} \fC_{de)a}$.  Setting $\widetilde{k} = -\frac{2}{7} \gamma_0$ and $A^a{}_b = 0$ yields the symmetry $u - \frac{2}{7} x^0 u_0 - \frac{3}{7} x^a u_a$.  Now relabel $\widetilde{k} = -\frac{3}{7} k$ and set $\gamma_0 = 0$ for the remaining symmetries.  The $\bV := \bV(0,t)$ case quickly follows from the condition $f_{u_0 u_i} = 0$.
 \end{proof}

 The number of symmetries in \eqref{E:curved-sym-1} is $\dim(M) + 1 + \dim(W) = 3n+1$.  By \eqref{E:curved-sym-2}, the linear transformation $\sfw_a \mapsto A^b{}_a \sfw_b$ must be a symmetry of $\fC \in S^3 W^*$, so from \S \ref{S:LM} is contained in $\ff^{\Ss}_0$.  Moreover, it must preserve the line $[\sfw^c] \in \bbP(W^*)$.  To maximize the solution space of \eqref{E:curved-sym-2}, we should examine the minimal $F_0$-orbit(s) in $\bbP(W^*)$.  (Recall that $W$ is reducible in type $B$ and $D$.)

 \begin{table}[h]
 \begin{align*}
 & \begin{array}{|c||c|c|c|c|c|c|c|c|} \hline
 G/P & B_\ell / P_2\,\,\, (\ell \geq 3) & D_\ell / P_2\,\,\, (\ell \geq 4) \\ \hline
 \fS_{(21)} & 2\ell^2 - 5\ell + 8 & 2\ell^2 - 7\ell + 11\\ \hline
 \fS_{(23)} & \begin{array}{ll} 2\ell^2 - 7\ell + 15, & \ell \geq 4; \\ 11, & \ell \geq 3 \end{array} 
 & \begin{array}{c} 2\ell^2 - 9\ell + 19 \\ \begin{tabular}{c} (same for $\fS_{(24)}$\\ when $\ell = 4$)\end{tabular} \end{array} \\ \hline
 \end{array} 
 &\begin{array}{|c||c|c|c|c|c|c|c|c|} \hline
 G/P & G_2 / P_2 & F_4 / P_1 & E_6 / P_2 & E_7 / P_1 & E_8 / P_8\\ \hline
 \fS & 7 & 28 & 43 & 76 & 147\\ \hline
 \end{array}
 \end{align*}
 \caption{Submaximal symmetry dimensions $\fS$ for $G$-contact structures (see \cite{KT2014})}
 \label{F:sm}
 \end{table}
 
 \begin{theorem}
 If $[\sfw^c]$ lies in a minimal $F_0$-orbit in $\bbP(W^*)$, then the $G$-contact structure \eqref{E:curved-G-contact} is submaximally symmetric.
 \end{theorem}

 \begin{proof} We give the proof for $\fg = E_8$ and $\fg = D_\ell$ $(\ell \geq 5)$.  The other cases are treated similarly.
 \begin{itemize}
 \item $\fg = E_8$, $\ff_0^{\Ss} = E_6$, $n = 28$: The $E_6$-highest weights of $W$ and $W^*$ are $\lambda_6$ and $\lambda_1$.  The dimension of $P_1 \subset E_6$ is 62.  Adding this to $3n+1 = 85$ agrees with $\fS = 147$.
 \item $\fg = D_\ell$, $\ff_0^{\Ss} = D_{\ell-3}$, $n = 2\ell-4$:  Here $3n+1 = 6\ell-11$, $m = 2\ell - 6$ and $W = \bbC^m \op \bbC \cong W^*$ as $D_{\ell-3}$-reps.  Also, $\dim(\fz(\ff_0)) = 2$, so one element acts non-trivially on $\fC$, and the other preserves it.  There are two harmonic curvature branches:
 \begin{enumerate}
 \item[(i)] $\bbC^m$: highest weight $\lambda_1$, so $(3n+1) + 1 + \dim(P_1) = (6\ell - 11) + 1 + (2\ell^2 - 15 \ell + 29) = 2\ell^2 - 9\ell + 19 = \fS_{(23)}$.
 \item[(ii)] $\bbC$: $(3n+1) + 1 + \dim(D_{\ell-3}) = 6\ell - 11 + 1 + (\ell-3)(2\ell-7) = 2\ell^2 - 7\ell + 11 = \fS_{(21)}$.
 \end{enumerate}
 \end{itemize}
 \end{proof}
  
 \subsection{Type A and C}
 \label{S:AC-submax}

 In Table \ref{F:AC-submax}, we give submaximally symmetric PDE (for $n \geq 2$) in the type A and C cases.  The $n=1$ cases are the classically known 2nd and 3rd order ODE cases.
 
 \begin{footnotesize}
 \begin{table}[h]
 \[
 \begin{array}{|c|c|c|c|c|c|} \hline
 G/P & \mbox{Invariants} & \mbox{Submax sym dim} & \mbox{Model} & \mbox{Symmetries}\\ \hline\hline
 %%%%%%
 \begin{array}{c} A_{n+1} / P_{1,n+1} \\ (n \geq 2) \end{array}
 & \begin{array}{l} \tau_E \neq 0,\\ \tau_F = 0,\\ \cW = 0 \end{array} & 
 \begin{array}{c}
 \fS_{(12)} = n^2 + 4 
 \end{array} & \begin{array}{c} u_{ij} = 0\\ \mbox{except}\\ u_{11} = x^n \end{array} &
 \begin{array}{cccc}
 1, \quad x^i, \quad u_k\,\, (k \neq n),\quad x^i u_l \,\, (1 \neq l \neq n),\\
 (x^1)^2 - 2 u_n, \quad x^1 u_n - \frac{1}{6} (x^1)^3, \\
 x^1 u_1 - 2 x^n u_n, \quad u - x^i u_i + x^1 u_1
 \end{array}
 \\ \hline 
 %%%%%%
 \begin{array}{c} A_{n+1} / P_{1,n+1} \\ (n \geq 2) \end{array}
 & \begin{array}{l} \tau_E = 0,\\ \tau_F = 0,\\ \cW \neq 0 \end{array} & 
 \begin{array}{c}
 \fS_{(1,n+1)} = n^2 + 4
 \end{array} & \begin{array}{c} u_{ij} = 0\\ \mbox{except}\\ u_{11} = (u_n)^2 \end{array} &
 \begin{array}{cccc}
  1, \quad x^k\,\,(k \neq n), \quad u_i, \quad x^k u_l \,\, (k \neq n; l \neq 1), \\
  u_n (x^1)^2 + x^n, \quad u - x^i u_i - \frac{1}{2} x^n u_n, \\
  x^1 u_1 + x^n u_n
 \end{array}
 \\ \hline
 %%%%%%
 \begin{array}{c} C_{n+1} / P_{1,n+1} \\ (n \geq 2) \end{array}
 & \begin{array}{l} \tau_E \neq 0,\\ \tau_{EV} = 0 \end{array} 
 & 
 \begin{array}{c}
 \fS_{(12)} = 2n^2 - n + 5
 \end{array}
 & \begin{array}{c} u_{ijk} = 0 \\ \mbox{except}\\ u_{111} = x^2 \end{array}
 & \begin{array}{c}
 1, \quad x^i, \quad u_k \,\, (k \neq 2), \quad x^i x^j, \quad x^i u_l\,\,l \geq 3,\\
 u_k u_l\, (k,l \geq 3), \quad
 u - x^2 u_2, \quad 6u_2 - (x^1)^3, \\
 24 x^1 u_2 - (x^1)^4, \quad
 x^1 u_1 - 3 x^2 u_2
 \end{array}\\ \hline
 %%%%%%
 \begin{array}{c} C_{n+1} / P_{1,n+1} \\ (n \geq 2) \end{array}
 & \begin{array}{l} \tau_E = 0,\\ \tau_{EV} \neq 0 \end{array} 
 & 
 \begin{array}{c} 
 \fS_{(1,n+1)} = \frac{3n^2+ n + 8}{2} 
 \end{array}
 & \begin{array}{c} u_{ijk} = 0 \\ \mbox{except}\\ u_{nnn} = u_{11} \end{array}
 & \begin{array}{c}
 1, \quad x^i, \quad u, \quad u_i, \quad
 x^k x^l \,\, (k \neq 1 \neq l), \quad \\
 x^k u_l \,\, (k \neq 1; l \neq n),\\
 3(x^1)^2 + (x^n)^3, \quad 3 x^1 u_1 + 2 x^n u_n - 5 u
 \end{array}\\ \hline
 \end{array}
 \]
 \caption{Submaximally symmetric PDE for $A_{n+1} / P_{1,n+1}$ and $C_{n+1} / P_{1,n+1}$ geometries}
 \label{F:AC-submax}
 \end{table}
 \end{footnotesize}

 \appendix 
 
 \section{Some explicit PDE $\cE$ for the flat $G$-contact structure}
 \begin{footnotesize}
 \[
 \begin{array}{cccc}
 G & B_\ell \mbox{ or } D_\ell & G_2 & D_4\\ \hline\hline
 W & \JS_m = \bbC^m \op \bbC & \GJ & \cJ_3(\underline{0}) \\ \hline
 \mbox{Coordinate} & (v,\lambda) & \lambda & \diag(t_1,t_2,t_3)\\ \hline
 (u_{ij}) & \left( \begin{array}{@{}c|c|c@{}}
 \langle v, v \rangle \lambda & \langle \cdot, v \rangle \lambda & \frac{1}{2}\langle v, v \rangle \\ \hline 
 \langle \cdot, v \rangle \lambda & \langle \cdot, \cdot \rangle \lambda & \langle \cdot, v \rangle\\ \hline
 \frac{1}{2} \langle v, v \rangle & \langle \cdot, v \rangle & 0
  \end{array} \right) & \left( \begin{array}{@{}c|c@{}}
 \frac{\lambda^3}{3} & \frac{\lambda^2}{2} \\ \hline 
 \frac{\lambda^2}{2} & \lambda \end{array} \right) & 
 \left( \begin{array}{@{}c|ccc@{}}
 2\lambda_1 \lambda_2 \lambda_3 & \lambda_2 \lambda_3 & \lambda_1 \lambda_3 & \lambda_1 \lambda_2\\ \hline 
 \lambda_2 \lambda_3 & 0 & \lambda_3 & \lambda_2\\ 
 \lambda_1 \lambda_3 & \lambda_3 & 0 & \lambda_1 \\
 \lambda_1 \lambda_2 & \lambda_2 & \lambda_1 & 0 \end{array} \right)
 \end{array}
 \]
 \end{footnotesize}

 Given $A = (a_{rs}) \in \tMat_{3 \times 3}(\bbC)$, let $\sfC_{rs}(A)$ be its $(r,s)$-th cofactor.
 \begin{itemize}
 \item $G = F_4$: $W = \cJ_3(\bbR_\bbC)$, so $a_{rs} = a_{sr}$.
 \begin{footnotesize}
 \begin{align*}
 (u_{ij}) =
 \left( \begin {array}{c|ccc|ccc} 
 \det(A) & \frac{1}{2} \sfC_{11}(A) & \frac{1}{2} \sfC_{22}(A) & \frac{1}{2} \sfC_{33}(A) & \sfC_{12}(A) & \sfC_{13}(A) & \sfC_{23}(A)\\ \hline
 \frac{1}{2} \sfC_{11}(A) & 0 &  \frac{1}{2} a_{33} &  \frac{1}{2} a_{22} & 0 & 0 & -a_{23}\\
 \frac{1}{2} \sfC_{22}(A) &  \frac{1}{2} a_{33} & 0 & \frac{1}{2} a_{11} & 0 & -a_{13} & 0 \\
 \frac{1}{2} \sfC_{33}(A) &  \frac{1}{2} a_{22} &  \frac{1}{2} a_{11} & 0 & -a_{12} & 0 & 0\\ \hline
 \sfC_{12}(A) & 0 & 0 & -a_{12} & -a_{33} & a_{23} & a_{13} \\ 
 \sfC_{13}(A) & 0 & -a_{13} & 0 & a_{23}& -a_{22} & a_{12} \\ 
 \sfC_{23}(A) & -a_{23} & 0 & 0 & a_{13} & a_{12} & -a_{11}\end {array}
 \right)
 \end{align*}
 \end{footnotesize}
 \item $G = E_6$: Since $W$ has weight $\lambda_1 + \lambda_1'$ for $\ff_0^{\Ss} = A_2 \times A_2$, we may use $W = \tMat_{3\times 3}(\bbC)$ as an alternative model to $W = \cJ_3(\bbC_\bbC)$.
 \begin{footnotesize}
 \begin{align*}
 (u_{ij}) = 
 \left( \begin {array}{c|ccc|ccc|ccc} 
 2 \det(A) & \sfC_{11}(A) & \sfC_{12}(A) & \sfC_{13}(A) & \sfC_{21}(A) & \sfC_{22}(A) & \sfC_{23}(A) & \sfC_{31}(A) & \sfC_{32}(A) & \sfC_{33}(A)\\ \hline
  \sfC_{11}(A) & 0 & 0 & 0 & 0 & a_{33} & -a_{32} & 0 & -a_{23} & a_{22} \\
  \sfC_{12}(A) & 0 & 0 & 0 & -a_{33} & 0 & a_{31} & a_{23} & 0 & -a_{21}\\
  \sfC_{13}(A) & 0 & 0 & 0 & a_{32} & -a_{31} & 0 & -a_{22} & a_{21} & 0\\ \hline
  \sfC_{21}(A) & 0 & -a_{33} & a_{32} & 0 & 0 & 0 & 0 & a_{13} & -a_{12}\\
  \sfC_{22}(A) & a_{33} & 0 & -a_{31} & 0 & 0 & 0 & -a_{13} & 0 & a_{11}\\
  \sfC_{23}(A) & -a_{32} & a_{31} & 0 & 0 & 0 & 0 & a_{12} & -a_{11} & 0\\ \hline
  \sfC_{31}(A) & 0 & a_{23} & -a_{22} & 0 & -a_{13} & a_{12} & 0 & 0 & 0\\
  \sfC_{32}(A) & -a_{23} & 0 & a_{21} & a_{13} & 0 & -a_{11} & 0 & 0 & 0\\
  \sfC_{33}(A) & a_{22} & -a_{21} & 0 & -a_{12} & a_{11} & 0 & 0 & 0 & 0
 \end {array}
 \right)
 \end{align*}
 \end{footnotesize}
 \end{itemize}

%%%%%%%%%%%%%%%%%%%%%%%%%%%%%%%%%%%%%%%%%%%%%%%%

 \end{document}